\documentclass{amsart}

\usepackage{amssymb}
\usepackage{amsmath}
\usepackage{amsfonts,dsfont}
\usepackage{graphicx}
\usepackage{latexsym}

\usepackage{hyperref}
\newcommand{\arxiv}[1]{\href{http://www.arXiv.org/abs/#1}{arXiv:#1}}
\usepackage{tikz}
\usetikzlibrary{shapes.symbols,arrows}

\newtheorem{theorem}{Theorem}[section]
\newtheorem{lemma}[theorem]{Lemma}
\newtheorem{proposition}[theorem]{Proposition}
\newtheorem{corollary}[theorem]{Corollary}
\theoremstyle{definition}
\newtheorem{definition}[theorem]{Definition}
\newtheorem{example}[theorem]{Example}

\theoremstyle{remark}
\newtheorem{remark}[theorem]{Remark}

\numberwithin{equation}{section}

\DeclareMathOperator{\ind}{ind}

\def\Nat{{\mathbb N}}

\def\wiel{\operatorname{Wi}}
\def\bzero{{\mathbf 0}}
\def\bunity{{\mathbf 1}}
\def\0{{\mathbf 0}}
\def\1{{\mathbf 1}}

\def\R{{\mathds R}}
\def\Rmax{\R_{\max}}
\def\Rp{\Rmax}
\def\Rpnn{\Rp^{n\times n}}

\def\digr{{\mathcal D}}

\def\crit{{\mathcal G}^c}

\def\subcrit{{\mathcal G}}
\def\mN{\mathcal{N}}

\def\hagr{{\mathcal G}^{ha}}
\def\ctgr{{\mathcal G}^{ct}}

\def\thrct{{\mathcal T}^{ct}}

\def\thrha{{\mathcal T}^{ha}}
\def\cycle{{Z}}

\newcommand{\walkslen}[3]{\mathcal{W}^{#3}(#1\to #2)}
\newcommand{\walks}[2]{\mathcal{W}(#1\to #2)}
\newcommand{\walksi}[1]{\mathcal{W}(#1\to)}
\newcommand{\walksilen}[2]{\mathcal{W}^{#2}(#1\to)}
\newcommand{\walkslennode}[4]{\mathcal{W}^{#3}(#1\xrightarrow{#4} #2)}

\newcommand{\walksnode}[3]{\mathcal{W}(#1\xrightarrow{#3} #2)}

\newcommand{\old}[1]{}

\def\circumf{\operatorname{cr}}
\def\cabdrive{\operatorname{cd}}
\def\ep{\operatorname{ep}}

\def\bnacht{B_{\operatorname{N}}}
\def\bharg{B_{\operatorname{HA}}}
\def\bct{B_{\operatorname{CT}}}
\def\g{\operatorname{g}}

\def\sqcup{\cup}
\def\bigsqcup{\bigcup}

\begin{document}

\title{Weak CSR Expansions and Transience Bounds in Max-Plus Algebra}


\author{Glenn Merlet}
\address{Glenn Merlet, Universit\'e d'Aix-Marseille, CNRS, IML, 13288 Marseille, France}
\email{glenn.merlet@gmail.com}

\author{Thomas Nowak}
\address{Thomas Nowak, \'{E}cole Normale Sup\'{e}rieure, F-75230 Paris Cedex 05, France}
\email{nowak@lix.polytechnique.fr}

\author{Serge{\u{\i}} Sergeev}
\address{Serge{\u\i} Sergeev, University of Birmingham, School of Mathematics,
Watson Building, Edgbaston B15 2TT, UK, and Moscow Centre for Continuous Mathematical Education, Russia.}
\email{sergiej@gmail.com}
\thanks{Th. Nowak was at LIX \'{E}cole Polytechnique and S. Sergeev 
was with the Max-Plus Team at INRIA and CMAP \'Ecole Polytechnique when this work was initiated.
The last author is supported by EPSRC grant RRAH15735, RFBR-CNRS grant 11-0193106 and
RFBR grant 12-01-00886.}





\date{}


\begin{abstract}
This paper aims to unify and extend existing techniques for deriving upper bounds on the transient of max-plus
matrix powers. To this aim, we introduce the concept of weak CSR expansions: $A^t=CS^tR\oplus B^t$. We observe that most of
the known bounds (implicitly)
take the maximum of (i) a bound for the weak CSR expansion to hold, which does not depend on the values of the entries of the matrix
but only on its pattern,
and (ii) a bound for the $CS^tR$ term to dominate.

To improve and analyze (i), we consider various cycle replacement techniques and show that some of the known bounds for indices and
exponents of digraphs apply here. We also show how to make use of various parameters of digraphs. To improve and analyze (ii), we
introduce three different kinds of weak CSR expansions (named after Nachtigall, Hartman-Arguelles, and Cycle Threshold).
As a result, we obtain a collection of bounds, in general incomparable to one another, but better than the bounds found in the literature.
\end{abstract}

\maketitle

\section{Introduction}

Max-plus algebra is a version of linear algebra developed over the max-plus semiring, which is the set $\Rmax=\R\cup\{-\infty\}$
equipped with the  multiplication $a \otimes b = a+b$ and the addition
$a\oplus b=\max(a,b)$. This semiring has zero $\bzero:=-\infty$ (neutral
with respect to $\oplus$) and unity $\bunity:=0$ (neutral with respect to
$\otimes$), and each element $\mu$ except for $\bzero$ has an inverse
$\mu^-:=-\mu$ satisfying $\mu\otimes \mu^-=\mu^-\otimes\mu=\bunity$. Taking powers
of scalars in $\Rmax$ means ordinary multiplication: $\lambda^{\otimes t}:=t\cdot\lambda$.

The max-plus arithmetic is extended to
matrices in the usual way, so that
$(A B)_{ij}=\bigoplus_k a_{ik}\otimes b_{kj}=\max_k (a_{ik}+ b_{kj})$ for matrices $A=(a_{ij})$ and $B=(b_{ij})$ of compatible sizes.
In this paper, all matrix multiplications are to be understood in the max-plus
sense.
For multiplication by a scalar and for taking powers of scalars we will
write the sign $\otimes$ explicitly, while for the matrix multiplication it will be always
omitted.

Historically, max-plus algebra first appeared
to analyze production systems driven by the dynamics
\begin{equation}
\label{e:dynam}
x_i(k+1)=\max_j \big( x_j(k) + a_{ij} \big) \enspace.
\end{equation}
Thus, repeated application of matrix $A=(a_{ij})$ in max-plus algebra to an
initial vector $x(0)$ computes the vectors $x(k)$.
Here $x(k)$ is typically a vector consisting of $n$ real components, expressing the times of
certain events happening during the $k$th production cycle.
According to dynamics~\eqref{e:dynam},
event~$i$ has to wait until all the preceding events~$j$ happen and the
necessary time delays~$a_{ij}$
have passed, so that event~$i$ can then occur as early as possible. Such
situation is usual
in train scheduling, working plan analysis, and synchronization of multiprocessor systems~\cite{BCOQ,But:10,HOW:05}.
Recently, Charron-Bost et al.~\cite{CBFWW11} have shown that also the behavior
of link reversal algorithms used for routing, scheduling, resource allocation,
leader election, and distributed queuing can be described by
a recursion of the form~\eqref{e:dynam}.

In this paper, we investigate the sequence of max-plus matrix powers $A^{t}
=\overbrace{AAA\cdots A}^{t\text{ times}}$.
Cohen et al.~\cite{CDQV-83} proved that this sequence eventually exhibits a periodic regime whenever~$A$ is
{\em irreducible}, i.e., whenever the digraph $\digr(A)$ described by~$A$ is strongly connected: there exists a positive integer~$\gamma$
and a nonnegative integer~$T$ such that
\begin{equation}\label{e:period}
\forall t\geq T\ :\quad A^{t+\gamma} = \lambda^{\otimes \gamma}\otimes A^t\,
\end{equation}
where~$\lambda=\lambda(A)$ is the unique max-plus eigenvalue of~$A$.
The smallest~$T$ that can be chosen in~\eqref{e:period} is called the {\em
transient} of~$A$; we denote it by~$T(A)$.

Since it satisfies $x(t)=A^tv$ every max-plus linear dynamical system, i.e., every sequence~$x(t)$ satisfying~\eqref{e:dynam} is periodic in the same sense whenever~$A$ is
irreducible.
Its transient~$T(A,v)$ in general depends on~$v$ and is always upper-bounded by~$T(A)$.

Bounds on the transients were obtained by Hartmann and Arguelles~\cite{HA-99}, Bouillard and Gaujal~\cite{BG-00},
Soto y Koelemeijer~\cite{SyK:03}, Akian et al.~\cite{AGW-05}, and Charron-Bost
et al.~\cite{CBFN-12}. Those bounds are incomparable because they depend on
different parameters of~$A$ or assume different hypotheses.
However they all appear, at least in the proofs,
as the maximum of a first bound independent of the values of the entries of~$A$
and a second bound taking those values into account.
The first motivation for this paper was to find a common ground for these bounds in order to
understand, unify, combine, and improve them.

Schneider~\cite{HS:PC} observed that the
Cyclicity Theorem can be written in the
form of a {\em CSR expansion}, which was formulated by Sergeev~\cite{Ser-09}:
there exists a nonnegative integer~$T$ such that
\begin{equation}
\label{csr}
\forall t\geq T \ :\quad A^t=\lambda^{\otimes t}\otimes CS^tR\enspace,
\end{equation}
where the matrices~$C$, $S$, and~$R$ are defined in terms of~$A$ and fulfill
$CS^{t+\gamma}R=CS^tR$ for all~$t\geq 0$.
In an earlier work, considering infinite-dimensional matrices, Akian, Gaubert and Walsh~\cite[Section 7]{AGW-05}
gave a similar formulation
originating in the preprints of Cohen~et~al.~\cite{CDQV-83}.

Because of the periodicity of the sequence $CS^tR$, the smallest~$T$ satisfying~\eqref{csr} is~$T(A)$.

Later, Sergeev and Schneider~\cite{SS-11} proved that for $t$ large enough, $A^t$ is the 
sum (in the max-plus sense) of terms of the form~$\lambda_i^{\otimes t}\otimes C_iS_i^tR_i$.
This sum, which we call {\em CSR decomposition}, has two remarkable properties:
it holds for reducible matrices as well as irreducible ones,
and the CSR decomposition holds for~$t\ge 3n^2$, a bound that does not depend on the values of the entries of~$A$.

As a common ground of transience bounds and $CSR$ decomposition, we propose the new concept of {\em weak CSR expansions}. We
suggest that all existing techniques
for deriving transience bounds implicitly use the idea that eventually we have
\begin{equation}
\label{weak-exp}
\forall t\geq T\ :\quad A^t=\big(\lambda^{\otimes t}\otimes
CS^tR\big)\oplus B^t
\enspace,
\end{equation}
where $C$, $S$, and $R$ are defined as in the CSR expansion, and~$B$
is obtained from~$A$ by setting several entries (typically, all
entries in several rows and columns) to~$\bzero$. In this case, we
say that~$B$ is {\em subordinate\/} to~$A$. Call the smallest~$T$
for which~\eqref{weak-exp} holds the {\em weak CSR threshold\/}
of~$A$ with respect to~$B$ and denote it by~$T_1(A,B)$.

This quantity heavily depends on the choice of~$B$, i.e., on which
entries are set to~$\bzero$.
If we choose~$B=(\bzero)$, then we recover the ordinary CSR
expansion and we have $T_1(A,B)=T(A)$.
If $\digr(B)$ is acyclic, then $B^n=(\bzero)$ and $T(A)\le\max(T_1(A,B),n)$.
More generally $T(A)\le\max(T_1(A,B),T_2(A,B))$, where $T_2(A,B)$ is the least integer satisfying
\begin{equation*}
\forall t\geq T\ :\quad \lambda^{\otimes t}\otimes\big(CS^tR\big)\ge B^t
\enspace.
\end{equation*}
Analogously, we call~$T_2(A,B,v)$ the least integer satisfying
\begin{equation*}
\forall t\geq T\ :\quad \lambda^{\otimes t}\otimes\big(CS^tRv\big)\ge B^tv
\enspace.
\end{equation*}

We claim that the bounds in~\cite{BG-00,CBFN-12,HA-99,SyK:03} implicitly are of this type, for various choices of~$B$ and various ways to bound
$T_1$ and~$T_2$.

We next summarize the contents of the remaining part of this paper.
In Section~\ref{s:prel}, we recall notions and results of max algebra, focusing on its relation to weighted digraphs.
In Section~\ref{s:csr}, we introduce three schemes of defining~$B$, and thereby weak CSR
expansions: the Nachtigall scheme, the Hartmann-Arguelles scheme, and the cycle threshold
scheme.
The first scheme is implicitly used in~\cite{AGW-05,BG-00,CBFN-12,SyK:03}, the second one is derived from~\cite{HA-99}, 
and the third one is completely new.

In Section~\ref{s:statements} we state some bounds on~$T_1(A,B)$ and~$T_2(A,B)$, thus on~$T$ that we obtain in this paper.
Those bounds strictly improve the ones in~\cite{BG-00,CBFN-12,HA-99,SyK:03}.
Moreover they can be combined in several ways.
Notably, for the three schemes defined in Section~\ref{s:csr}, we bound the weak CSR threshold~$T_1(A,B)$ by the
{\em Wielandt number\/}
\begin{equation}\label{e:wielandt}
\wiel(n)=\begin{cases} 0 & \text{if } n=1\\(n-1)^2+1 & \text{if }n>1\end{cases}
\end{equation}
(named in honor of~\cite{Wie-50}).
The bound $\wiel(n)$ is optimal because it is the worst case transient of
powers of Boolean matrices,
i.e., matrices with entries~$\bzero$ and~$\bunity$
(see
Remarks~\ref{r:Boolean} and~\ref{r:Boolean2}).
We also recover another optimal bound for Boolean matrices due to Dulmage and
Mendelsohn~\cite{DM-64} that do not only depend on~$n$ but also on some graph parameter.
The section also includes a examples to compare the different bounds.

In Section~\ref{s:comparison}, we compare our results to some bounds found in the literature.

In Section~\ref{s:ProofS}, we explain the strategy of the proof, which leads us to introduce a graph theoretic quantity,
which we name {\em cycle removal threshold} of a graph and state bounds on~$T_1(A,B)$ that depend on this quantity for some graphs.

In Sections~\ref{s:CsrtoWalk} and~\ref{s:TcrToT1}, we prove the results stated in Section~\ref{s:ProofS} to bound~$T_1(A,B)$ in terms
of the cycle removal threshold.

In Section~\ref{s:Tcr} we bound the cycle removal threshold.
First we recall the bounds of~\cite{CBFN-11} that depend on several parameters of~$\digr(A)$ and use the ideas of
Hartman and Arguelles~\cite{HA-99}
to give a new bound depending on less parameters.
Then, we introduce a new technique leading to other two bounds on~$T_1(A,B)$.

In Section~\ref{s:TcrToT2} we prove the bounds on~$T_2(A,B)$.

In Section~\ref{s:Ep}, we recall some bounds on the index of Boolean matrices to be used in some bounds on~$T_1$.

The technique of local reduction, originating from Akian, Gaubert and Walsh~\cite[Section 7]{AGW-05},
is recalled in Section~\ref{s:localred}. We show that this
technique can be combined with any of the CSR schemes described in Section~\ref{s:csr}.

\section{Preliminaries}\label{s:prel}


\subsection{Walks in weighted digraphs}

Let us recall the optimal walk interpretation of matrix powers in max algebra.
This is the fact that the entries of a matrix power~$A^t$ are equal to maximum
weights of walks of length~$t$ in the digraph associated to matrix~$A$.

To a square matrix $A=(a_{ij})\in\Rp^{n\times n}$ we associate an edge-weighted
digraph $\digr(A)$ with set of nodes $N=\{1,2,\dots,n\}$ and set of
edges~$E\subseteq N\times N$ containing a pair~$(i,j)$ if and only
if~$a_{ij}\neq\bzero$; the weight of an edge~$(i,j)\in E$ is defined to
be~$p(i,j)=a_{ij}$.
A {\em walk} $W$ in $\digr(A)$ is a finite sequence $(i_0,i_1,\ldots i_L)$ of adjacent nodes of~$\digr(A)$.
We define its length $l(W)=L$ and weight $p(W)=a_{i_0,i_1} \otimes a_{i_1,i_2} \cdots\otimes a_{i_{t-1},i_t}$.
A closed walk is a walk whose start node~$i_0$ coincides with its end node~$i_L$. Closed walks are often called circuits in the literature.
There exists an empty closed walk at every node of length~$0$ and
weight~$0$.

The multiplicity of an edge~$e$ in~$W$ is the number of~$k$'s such that $(i_k,i_{k+1})=e$.
A {\em subwalk \/} of walk~$W$ is a walk~$V$ such that the edges of~$V$ appear in~$W$ with larger multiplicity.
A subwalk of~$W$ is a {\em proper\/} subwalk if it is not equal to~$W$.

A closed walk is a {\em cycle\/} if it does not contain any nonempty closed walk as a proper subwalk.
A walk is a  {\em path\/} if it does not contain a nonempty cycle as a subwalk.

An elementary result of graph theory states that a walk can always be split into a path and some cycles.
Reciprocally, union of edges of one path and some cycles can always be reordered into a walk provided the graph with all the edges and nodes of those walks is connected.
The best way to see this is in terms of multigraph~$M(W)$ defined by a
walk~$W$.

For a set~$\mathcal{W}$ of walks, we write~$p(\mathcal{W})$ for the supremum of walk weights
in~$\mathcal{W}$.
Denote by~$\walkslen{i}{j}{t}$ the set of all walks from~$i$ to~$j$ of length~$t$
and write $A^t = (a_{ij}^{(t)})$.
It is immediate from the definitions that
\begin{equation}\label{e:walksense1}
a_{ij}^{(t)} = p\left( \walkslen{i}{j}{t} \right) \enspace.
\end{equation}

When we do not want to restrict the lengths of walks, we define the set~$\walks{i}{j}$ of all walks connecting~$i$ to~$j$.
An analog of $(I-A)^{-1}$ in max-plus algebra is the {\em Kleene star\/}
\begin{equation}\label{klsdef}
A^*=I\oplus A\oplus A^2\oplus A^3 \oplus \ldots \enspace,
\end{equation}
where~$I$ is the max-plus identity matrix.
It follows from the optimal walk interpretation~\eqref{e:walksense1} that series~\eqref{klsdef} converges if and only if $p(\cycle)\leq0$ for all closed walks~$\cycle$ in~$\digr(A)$, in which case it can be truncated as $A^*=I\oplus A\oplus\ldots\oplus A^{n-1}$.
If we denote $A^* = (a_{ij}^*)$, it is again immediate that
\begin{equation}
\label{e:walksense2}
a_{ij}^* =  p\big(\walks{i}{j} \big) \enspace.
\end{equation}

The {\em maximum cycle mean\/} of~$A\in\Rpnn$ is defined by
\begin{equation}\label{mcgm}
\begin{split}
\lambda(A)&=\max\{ p(\cycle)^{\otimes1/l(\cycle)} \mid \cycle \text{ is a nonempty
cycle in }\digr(A)\}
\enspace.
\end{split}
\end{equation}
Because every closed walk is composed of cycles, we could replace
``cycle'' by ``closed walk'' in definition~\eqref{mcgm}. The maximum
cycle mean~$\lambda(A)$ is equal to the greatest max-algebraic
eigenvalue of~$A$, i.e., a~$\mu\in\Rp$ such that there exists a
nonzero vector~$x$ satisfying $A\otimes x=\mu\otimes x$. Nonempty
closed walks of weight $\lambda(A)$ are called {\em critical}, and
so are the nodes and edges on these wlks. The subgraph
of~$\digr(A)$ consisting of the set of critical nodes~$N_c$ and the
set of critical edges~$E_c$ is called the {\em critical graph\/} of
$A$ and is denoted by $\crit(A)=(N_c,E_c)$. A useful fact (used
throughout the paper) is that every nonempty closed walk
in~$\crit(A)$ is critical.


As we will see, the behavior of max-algebraic matrix powers is eventually
dominated by the
walks that visit the critical graph.
The set of such walks in $\walkslen{i}{j}{t}$
will be denoted by $\walkslennode{i}{j}{t}{\crit}$
More generally, for a node~$k$ and a subgraph $\digr$ of~$\digr(A)$ we write
$$\walkslennode{i}{j}{t}{k}=\bigcup_{t_1+t_2=t} \left\{W_1W_2 |W_1\in\walkslen{i}{k}{t_1}, W_2\in\walkslen{k}{j}{t_2} \right\},$$
$$\quad \walkslennode{i}{j}{t}{\digr}=\bigcup_{k\in\digr}
\walkslennode{i}{j}{t}{k} \quad \textnormal{and} \quad \walksnode{i}{j}{\digr} =
\bigcup_{t\geq0} \walkslennode{i}{j}{t}{\digr}\enspace.$$

\old{
We will also consider orbits $\big(A^t v\big)_{t\geq0}$ starting at
some initial vector~$v\in\Rp^n$, it will be convenient to modify the weights of
walks, multiplying them by the component of $v$ corresponding to the walk's end
node. Thus for a walk~$W$ in~$\digr(A)$ ending in node~$j$, we define $p_v(W)=
p(W)\otimes v_{j}$. For a set~$\mathcal{W}$ of walks, we will likewise write $p_v(\mathcal{W})$
for the maximum~$p_v(W)$ where~$W\in\mathcal{W}$. Here we are interested in the set of
all walks of length $t$ starting at a node $i$.  The set of such walks will be
denoted by $\walksilen{i}{t}$, or just $\walksi{i}$ if their length not
restricted.  It follows directly from the definitions that
\begin{equation}
\label{e:v-walks}
(A^t v)_i = p_v\big( \walksilen{i}{t} \big)\quad\text{and}\quad
(A^* v)_i = p_v\big( \walksi{i} \big)\enspace.
\end{equation}}

\subsection{Cyclicity of digraphs}
A digraph $\subcrit=(N,E)$ is {\em strongly connected\/} if there exists a walk
from~$i$
to~$j$ for all nodes $i,j\in N$.
A {\em strongly connected component (s.c.c.)\/} of~$\subcrit$ is a maximal strongly connected subgraph
of~$\subcrit$.
Digraph~$\subcrit$ is called {\em completely reducible} if there are no edges between
distinct s.c.c.'s of~$\subcrit$.
 The critical graph $\crit(A)$ will be the most important example of this.

Matrix $A$ is called {\em irreducible\/} if its associated digraph
is strongly connected, and {\em reducible\/} otherwise. Further, it
is called completely reducible if so is the associated digraph.

The {\em cyclicity}~$\gamma(\subcrit)$ of a strongly connected digraph~$\subcrit$ is the greatest common divisor of the lengths of its closed walks.
If~$\subcrit$ is not strongly connected, its cyclicity~$\gamma(\subcrit)$ is
the least common multiple of the cyclicities of its s.c.c.'s.
It is well-known that any two lengths of walks on~$\subcrit$ both starting at some node~$i$ and both ending at some node~$j$
are congruent modulo~$\gamma(\subcrit)$.
Moreover, if~$\subcrit$ is strongly connected, there is walk from~$i$ to~$j$
of all lengths that are large enough and that are congruent to some~$t_{ij}$ modulo~$\gamma(\subcrit)$.

We call a subgraph~$\subcrit$ of~$\crit(A)$ a {\em representing subgraph\/} if~$\subcrit$ is completely reducible and every s.c.c.\ of~$\crit(A)$ contains exactly one s.c.c.\ of~$\subcrit$.
The cyclicity~$\gamma(\subcrit)$ of a representing subgraph of~$\crit(A)$ is always a multiple of~$\gamma\big(\crit(A)\big)$.
Hence Equation~\eqref{e:period} also holds with~$\gamma=\gamma(\subcrit)$ instead of~$\gamma\big(\crit(A)\big)$.

\old{
Because of the walk interpretation~\eqref{e:walksense1} below, the sequence of
digraphs~$\digr(A^t)$ becomes periodic whenever~$A$ is irreducible, i.e.,
$\digr(A)$ is strongly connected.
Its minimal period is the cyclicity of~$\digr(A)$.

The first bound on the index of a digraph was given by Wielandt:
\begin{equation*}
\ind(A) \leq \wiel(n) = n^2 - 2n + 2
\end{equation*}
where $\ind(A)$ denotes the index of~$\digr(A)$.
It is optimal is the sense that there is a matrix of size~$n$ with index~$\wiel(n)$ for any~$n$.}

\subsection{Visualization and max-balancing}\label{ss:vis}

The maximum cycle mean $\lambda(A)$ also appears as the least $\mu\in\Rp$ such that there
exists a finite vector~$x$ satisfying
$Ax\leq\mu\otimes x$. When $\mu=\lambda(A)$, we can take
$x_i=\bigoplus_{j=1}^n (\lambda^-(A)\otimes A)^*_{ij}$, that is, the
component-wise maximum of all columns of $(\lambda^-(A)\otimes
A)^*$. Setting $D$, resp.\ $D^-$, to be the diagonal matrix with
entries $d_{ii}=x_i$ and $d_{ij}=\bzero$ for $i\neq j$, resp.\
$d^-_{ii}:=x_i^-$ and $d_{ij}=\bzero$ for $i\neq j$, we obtain for
$B=D^-(\lambda^-\otimes A)D$ that $\crit(B)=\crit(A)$ and
\begin{equation}
\label{visdef}
\begin{split}
& b_{ij}\leq \bunity\quad\text{for all $i,j$}\enspace,\\
& b_{ij}=\bunity\quad\text{for all } i,j \text{ in }\crit(B) \enspace.
\end{split}
\end{equation}
When~\eqref{visdef} holds we say that $B$ is {\em visualized}: it
exhibits the edges of the critical graph. A diagonal matrix~$X$ such
that $B=D^-(\lambda^-(A)\otimes A)D$ is visualized and
$\crit(B)=\crit(A)$ is also called a Fiedler-Pt\'ak
scaling~\cite{FP-67} of~$A$. In this case, we call~$B$ a {\em
visualization\/} of~$A$.

Fiedler-Pt\'ak scalings were described in more detail by Sergeev,
Schneider and Butkovi\v{c}~\cite{SSB} using Kleene stars and max
algebra. Butkovi{\v c} and Schneider~\cite{BS-05} described
applications to various kinds of nonnegative similarity scalings. A
Fiedler-Pt\'ak scaling, particularly interesting to us, is called
the {\em max-balancing}. It was described by Schneider and
Schneider~\cite{SS-91}:

\begin{theorem}[Schneider and Schneider~\cite{SS-91}]
\label{schsch} For all $A\in\Rp^{n\times n}$ exists a
visualization~$B$ of~$A$ satisfying the following equivalent
properties:
\begin{enumerate}
\item (Cycle cover) For all edges $(i,j)$ in~$\digr(B)$ there exists a cycle~$\cycle$ in~$\digr(B)$ containing~$(i,j)$
such that all edges of~$\cycle$ have weight at least~$b_{ij}$.
\item (Max-balancing) For all sets $M\subseteq\{1,\ldots,n\}$, we have:
$\displaystyle\max_{i\in M, j\not\in M} b_{ij} = \max_{i\not\in M, j\in M} b_{ij}.$
\end{enumerate}
\end{theorem}

\subsection{CSR expansions, Weak CSR expansions}

For any~$A\in\Rpnn$ and  any subgraph~$\subcrit$ of~$\crit(A)$ with no trivial s.c.c., we
set $M=\left((\lambda(A)^- \otimes A\big)^{\gamma(\subcrit)}\right)^*$ and define the matrices
$C,S,R\in\Rpnn$ by
\begin{equation}
\label{csrdef}
\begin{split}
c_{ij}&=
\begin{cases}
m_{ij} &\text{if $j$ is in $\subcrit$}\\
\bzero &\text{otherwise,}
\end{cases}\quad
r_{ij}=
\begin{cases}
m_{ij} &\text{if $i$ is in $\subcrit$}\\
\bzero &\text{otherwise,}
\end{cases}\\
s_{ij}&=
\begin{cases}
\lambda(A)^- \otimes a_{ij} &\text{if $(i,j)$ is in $\subcrit$}\\
\bzero &\text{otherwise.}
\end{cases}
\end{split}
\end{equation}
When the dependency on~$\subcrit$ needs to be emphasized, we write $C_\subcrit$, $S_\subcrit$ and~$R_\subcrit$ instead of~$C$, $S$ and~$R$.

Essentially $C$ and $R$ can be regarded as sub-matrices of $M$ extracted from the columns,
resp.\ the rows of $M$ with indices in $\subcrit$.
If~$A$ is visualized, then matrix $S$ is exactly the associated
Boolean matrix of $\subcrit$.
The matrices $C$, $S$, and $R$ are called the {\em CSR terms\/} of~$A$ with respect to~$\subcrit$.


The following is a CSR version of the Cyclicity Theorem.

\begin{theorem}[\cite{HS:PC,Ser-09}]\label{schneider}
Let $A\in\Rpnn$ be irreducible and let $C,S,R$ be the CSR terms of~$A$ with respect to~$\crit(A)$. Then for all~$t\geq T(A)$:
\[A^t=\lambda(A)^{\otimes t} \otimes CS^tR \]
\end{theorem}

As it is shown below, Theorem~\ref{schneider} also holds with~$\crit(A)$ replaced by
some representing subgraph~$\subcrit$ of~$\crit(A)$.

Note that this theorem implies periodicity of~$A^t$ after~$T(A)$, because the
sequence of matrices~$CS^tR$ is {\em purely periodic}, i.e., periodic from the
very beginning. In fact, this statement is more generally true for all
completely reducible (and hence also for all representing) subgraphs of~$\crit(A)$:

\begin{proposition}[\cite{SS-11}]
\label{p:purely}
Let~$A\in\Rpnn$ be irreducible and $C,S,R$ be the CSR terms of~$A$ with respect
to some completely reducible subgraph~$\subcrit$ of the critical graph~$\crit(A)$.
Then the sequence of matrices~$CS^tR$ is purely periodic.
\end{proposition}

This fact was shown by Sergeev and Schneider~\cite{SS-11}, where CSR
terms with respect to completely reducible subgraphs of $\crit(A)$
were studied in detail. It can also be deduced from
Theorem~\ref{t:representation} proved below.



A {\em weak CSR expansion\/} of~$A$ is an expansion of the
form~\eqref{weak-exp} where $C,S,R$ are CSR terms with respect to
some representing subgraph of~$\crit(A)$ and~$\digr(B)$ is a
subgraph of~$\digr(A)$ disjoint to~$\crit(A)$. In particular, the
result of Theorem~\ref{schneider} is also a weak CSR expansion
(take~$B$ equal to the max-plus zero matrix).

By iteration of weak CSR expansions, we recover the CSR
decomposition of~$A$ introduced in~\cite{SS-11}. Bounds on~$T_1$
give bounds on the time from which $A^t$ admits such a
decomposition. (See Corollary~\ref{c:T1ha}).

\section{Weak CSR Schemes}\label{s:csr}

In this section, we introduce our three schemes for weak CSR expansions and
discuss their relation.
We define them in terms of the subgraph~$\digr$ of~$\digr(A)$ whose edges denote the
indices that are set to~$\bzero$ in the subordinate matrix~$B$.
More explicitly:
\begin{equation}\label{e:CSRschemes}
b_{ij} =
\begin{cases}
\bzero & \text{if $i$ or $j$ is a node of $\digr$} \\
a_{ij} & \text{else}
\end{cases}
\end{equation}

The three schemes are:

\begin{enumerate}
\item {\em Nachtigall scheme}. Here, the subgraph $\digr=\crit(A)$.
We denote the resulting matrix~$B$ by $\bnacht$.

This scheme is consistent with the expansions introduced by
Nachtigall~\cite{Nacht}, which was studied by
Moln\'arov\'a~\cite{Mol-03} and Sergeev and Schneider~\cite{SS-11}. It was
used by almost all authors who studied matrix transients \cite{AGW-05,BG-00,CBFN-12,SyK:03}, excluding
Hartmann and Arguelles~\cite{HA-99}.

\item {\em Hartmann-Arguelles scheme}. This scheme is defined in terms of the
max-balancing~$V=(v_{ij})$ of~$A$.
Given $\mu\in\Rp$, define the {\em
Hartmann-Arguelles threshold graph} $\thrha(\mu)$ induced by all
edges~$(i,j)$ in $\digr(A)=\digr(V)$
with $v_{ij}\geq\mu$.
For $\mu=\lambda(A)=\lambda(V)$ we have $\thrha(\mu)=\crit(A)=\crit(V)$.
Let $\mu^{ha}$ be the maximum of $\mu\leq\lambda(A)$ such that $\thrha(\mu)$ has a s.c.c.\
that does not contain any s.c.c.\ of $\crit(A)$.
If no such~$\mu$ exists, then $\mu^{ha}=\bzero$ and $\thrha(\mu^{ha})=\digr(A)$.

The subgraph $\digr=\hagr$ defining~$B$ in the Hartmann-Arguelles scheme is the union
of the s.c.c of $\thrha(\mu^{ha})$ intersecting~$\crit(A)$.
We denote this matrix~$B$ by $\bharg$.
Observe that $\lambda(\bharg)=\mu^{ha}$
and the graphs $\thrha(\mu)$, for all $\mu$, are completely
reducible due to max-balancing (more precisely, the cycle cover
property).


\item {\em Cycle threshold scheme}.
For $\mu\in\Rp$, define the
{\em cycle threshold graph} $\thrct(\mu)$ induced by all nodes
and edges belonging to the cycles in $\digr(A)$ with mean weight greater or
equal to~$\mu$.
Again, for $\mu=\lambda(A)$ we have
$\thrct(\mu)=\crit(A)$.
Let $\mu^{ct}$ be the maximum of $\mu\leq\lambda(A)$ such that $\thrct(\mu)$ has a s.c.c.\
that does not contain any s.c.c.\ of $\crit(A)$.
If no such~$\mu$ exists, then $\mu^{ct}=\bzero$ and $\thrct(\mu^{ct})$ is
equal to~$\digr(A)$.

The subgraph $\digr=\ctgr$ defining~$B$ in the cycle threshold scheme is the union
of the s.c.c of~$\thrct(\mu^{ct})$ intersecting~$\crit(A)$.
This matrix~$B$ will be denoted by $\bct$.
We again observe that $\lambda(\bct)=\mu^{ct}$.
\end{enumerate}

\begin{remark}\label{r:Boolean}
Since $\crit(A)\subset\digr$, we see that $\lambda(B)<\lambda(A)$.

In particular, if~$A$ is an irreducible Boolean matrix,
then $\crit(A)=\digr(A)$ and~$B=(\bzero)$
for all schemes,
thus~$T_1(A,\bnacht)=T_1(A,\bharg)=T_1(A,\bct)=T(A)$.

The weak CSR thresholds hence are generalizations of the transient
of irreducible Boolean matrices, which has been investigated in the
literature under the name {\em index (of convergence)\/} of~$A$, or also {\em
exponent\/} in case of primitive matrices. See Section~\ref{s:Ep}
for a brief account.
\end{remark}

\begin{proposition}
The matrices~$\bnacht$ and~$\bharg$ can be computed in polynomial time.
The computation of the threshold graphs~$\thrct(0)$ is NP-hard.
\end{proposition}
\begin{proof}
The computation of $\bnacht$ relies on the computation of
$\crit(A)=(N_c,E_c)$, for which we can exploit the well-known
criterion $a_{ij}a^*_{ji}=\1\Leftrightarrow (i,j)\in E_c(A)$  (when
$\lambda(A)=\1$). This yields complexity at most $O(n^3)$.

Concerning $\bharg$, Schneider and Schneider~\cite{SS-91} proved
that a max-balancing of~$A$ can be computed in polynomial time (at
most $O(n^4)$). The same order of complexity is added if we
``brutally'' examine at most $n^2$ threshold graphs (for each of
them, the strongly connected components found in $O(n^2)$ time). A
better complexity result can be derived from the work of
Hartmann-Arguelles~\cite{HA-99}.

To show NP-hardness of the computation of~$\thrct(\mu)$, we reduce
the Longest Path Problem~\cite[p.~213, ND29]{GJ-79} to it. Consider the
Longest Path Problem as a decision problem that takes as input an
edge-weighted digraph with integer weights, a pair of nodes $(i,j)$ with $i\neq j$
in the digraph, and an integer~$K$. The output is YES if there
exists a path of weight at least~$K$ from~$i$ to~$j$. The output is
NO if there is none. 
Observe that if $i\neq j$, then by inserting
the edge $(j,i)$ with weight $-K$, the Longest Path Problem can be
polynomially reduced to the problem of calculating $\thrct(0)$ by
checking whether the new edge~$(j,i)$ belongs to $\thrct(0)$.
\end{proof}

The relation between these schemes is as follows. The cycle
threshold scheme is most precise, while the Nachtigall scheme
is the coarsest. We measure this in terms of the size of $B$ and the value
$\lambda(B)$.

\begin{proposition}\label{exp-relations}
$\bct$ is subordinate to $\bharg$, which is subordinate
to $\bnacht$. In particular,
$$\lambda(\bct)\leq\lambda(\bharg)\leq\lambda(\bnacht)\enspace.$$
\end{proposition}
\begin{proof} Evidently both $\digr(\bct)$ and $\digr(\bharg)$ are subgraphs of
$\digr(\bnacht)$, which is extracted from all non-critical nodes.
This implies $\lambda(\bct)\leq\lambda(\bnacht)$ and
$\lambda(\bharg)\leq\lambda(\bnacht)$.

We show that $\digr(\bct)$ is a subgraph of $\digr(\bharg)$. For
this we can assume that the whole digraph is max-balanced, and
notice first that $\thrha(\mu)\subseteq\thrct(\mu)$ for any value of
$\mu$. We also have that $\thrha(\mu_1)\supseteq\thrha(\mu_2)$ and
$\thrct(\mu_1)\supseteq\thrct(\mu_2)$ for any $\mu_1\leq\mu_2$. Now
consider the value $\mu^{ct}$. The components of $\thrct(\mu^{ct})$
which do not contain the components of $\crit(A)$, have the property
that any other cycle intersecting with them has a strictly smaller
cycle mean. It follows that all edges of these components have cycle
mean $\mu^{ct}$. Indeed, suppose that there is a component
containing an edge with a different weight. In this component, any
cycle that contains this edge also has an edge with weight strictly
greater than $\mu^{ct}$. The cycle cover property implies that there
is a cycle containing this edge, where this edge has the smallest
weight. The mean of that cycle is strictly greater than $\mu^{ct}$,
a contradiction. But then $\thrha(\mu^{ct})$ contains these
components as its s.c.c.'s. In particular they do not contain the
components of $\crit(A)$, hence $\mu^{ct}\leq\mu^{ha}$.

If $\mu:=\mu^{ct}=\mu^{ha}$ then $\thrha(\mu)\subseteq\thrct(\mu)$, while we have shown that the components of $\thrct(\mu)$ not
containing the components of $\crit(A)$ are also components of $\thrha(\mu)$. It follows that $\hagr\subseteq\ctgr$.

If $\mu^{ct}<\mu^{ha}$ then we obtain that
$$\ctgr\supseteq\thrct(\mu^{ha})\supseteq\thrha(\mu^{ha})\supseteq\hagr,$$
thus $\hagr\subseteq\ctgr$ in any case, hence $\digr(\bct)\subset\digr(\bharg)$.
\end{proof}

The following example shows that all three schemes can differ and, moreover, that the thresholds $T_1(A,\bnacht),$
$T_1(A,\bharg)$ and $T_1(A,\bct)$ can all differ.

\begin{example}\label{ex:lambdas}

Consider a matrix
\begin{equation}
\label{e:gl-matx}
A=
\begin{pmatrix}
0 & 0  &  -1 &  -\infty & -7 \\
0 & 0  &  -1 &  -\infty &  -7 \\
-1 & -1 &  -1  & -3 &  -7 \\
-3 &  -\infty  & -\infty &  -2 & -7 \\
-7 & -7 & -7 & -7 & -3
\end{pmatrix}
\end{equation}
In this example we have $\lambda(A)=0$, it is visualized and, moreover,
max-balanced. The matrices $\bnacht$, resp. $\bharg$ and
$\bct$ are formed by setting the first $2$ rows and columns, resp. the first $3$ and $4$
rows and columns to $\bzero=-\infty$, and the corresponding values are
$\lambda(\bnacht)=-1$, $\lambda(\bharg)=-2$ and $\lambda(\bct)=-3$.
The corresponding thresholds are $T_1(A,\bnacht)=2$, $T_1(A,\bharg)=3$ and
$T_1(A,\bct)=4$: all different. The periodicity threshold of $(A^{\otimes t})_{t\geq 1}$
is equal to $T(A)=5$, which is the same as $T_2(A,\bnacht)=T_2(A,\bharg)=T_2(A,\bct)$.

Let us provide a class of examples that generalizes the example above to arbitrary dimension. For any matrix $A$ in this
class of examples, all three schemes are different but the corresponding thresholds $T_1(A,B)$ may coincide.

Consider a matrix $A$ such that
the node set $N$ of $\digr(A)$ is partitioned into
$N=N_c\cup N_n\cup N_{ha}\cup N_{ct}$, see Figure~\ref{f:incompar}.
For each $x\in\{c,n,ha,ct\}$, the nodes in $N_x$ form a strongly connected graph
where all edges have weight $\lambda_x$. We set $\lambda_c>\lambda_n>\lambda_{ha}>\lambda_{ct}$.
For each set $N_x$ with $x\in\{n,ha,ct\}$, we assume that there is at least one edge from $N_x$ to some
set $N_y$ with $\lambda_y>\lambda_x$, and one edge from one of such $N_y$ to $N_x$.
With this assumption, it can be shown that $\digr(A)$ is strongly connected.
Let us also assume that all such edges (from $N_x$ and to $N_x$) have the same weight $\delta_x$.
Observe that for the matrix of~\eqref{e:gl-matx}, we have $N_c=\{1,2\}$, $N_n=\{3\}$,
$N_{ha}=\{4\}$ and $N_{ct}=\{5\}$; $\lambda_c=0$, $\lambda_n=-1$, $\lambda_{ha}=-2$ and $\lambda_{ct}=-3$;
$\delta_n=-1$, $\delta_{ha}=-3$ and $\delta_{ct}=-7$.

\begin{figure}
\centering
\begin{tikzpicture}[>=latex',scale=1.5,thick]
\begin{scope}[shift={(0,0)}]
\node[shape=circle,draw] (c1) at (0,0) {};
\node[shape=circle,draw] (c2) at (0.8,-0.7) {};
\node[shape=circle,draw] (c3) at (0.3,-1.5) {};
\node[shape=circle,draw] (c4) at (-0.6,-1.5) {};
\node[shape=circle,draw] (c5) at (-0.8,-0.6) {};

\draw[->,dashed] (c1) to [out=-20,in=120]
node[below]{$\scriptstyle\lambda_c$} (c2);
\draw[->,dashed] (c2) to [out=170,in=0] node[above]{$\scriptstyle\lambda_c$} (c5);
\draw[->,dashed] (c5) to [out=90,in=180] node[left]{$\scriptstyle\lambda_c$} (c1);

\draw[->,dashed] (c3) to [out=200,in=-20]
node[above]{$\scriptstyle\lambda_c$} (c4);
\draw[->,dashed] (c4) to [out=120,in=-90]
node[right]{$\scriptstyle\lambda_c$} (c5);
\draw[->,dashed] (c5) to [out=-20,in=100] node[below]{$\scriptstyle\lambda_c$} (c3);

\node[cloud, cloud puffs=24, draw,minimum width=3.5cm, minimum
height=3.5cm,thin] at (-0.05,-0.8) {};

\node at (-1.3,-0.2) {$N_c$};
\end{scope}

\begin{scope}[shift={(3,0)}]
\node[shape=circle,draw] (n1) at (0,0) {};
\node[shape=circle,draw] (n2) at (0.8,-0.7) {};
\node[shape=circle,draw] (n3) at (0.3,-1.5) {};
\node[shape=circle,draw] (n4) at (-0.8,-1.2) {};

\draw[->,dashed] (n1) to [out=-20,in=120]
node[below]{$\scriptstyle\lambda_n$} (n2);
\draw[->,dashed] (n2) to [out=-90,in=40]
node[left]{$\scriptstyle\lambda_n$} (n3);
\draw[->,dashed] (n3) to 
node[left]{$\scriptstyle\lambda_n$} (n1);
\draw[->,dashed] (n3) to [out=200,in=-40]
node[above]{$\scriptstyle\lambda_n$} (n4);
\draw[->,dashed] (n4) to [out=90,in=-150]
node[right]{$\scriptstyle\lambda_n$} (n1);

\node[cloud, cloud puffs=24, draw,minimum width=3.5cm, minimum
height=3.5cm,thin] at (0,-0.8) {};

\node at (1.3,-0.2) {$N_n$};
\end{scope}

\begin{scope}[shift={(0,-3)}]
\node[shape=circle,draw] (h1) at (0,0) {};
\node[shape=circle,draw] (h2) at (-0.5,-1.4) {};
\node[shape=circle,draw] (h3) at (-1.7,-0.6) {};

\draw[->,dashed] (h1) to [out=-100,in=40]
node[left]{$\scriptstyle\lambda_{ha}$} (h3);
\draw[->,dashed] (h2) to [out=-190,in=-60]
node[right]{$\scriptstyle\lambda_{ha}$} (h3);
\draw[->,dashed] (h3) to [out=0,in=90]
node[right]{$\scriptstyle\lambda_{ha}$} (h2);

\draw[->,dashed] (h3) to [out=80,in=160]
node[below]{$\scriptstyle\lambda_{ha}$} (h1);

\node[cloud, cloud puffs=24, draw,minimum width=3.7cm, minimum
height=3.7cm,thin] at (-0.8,-0.5) {};

\node at (-1.6,0.7) {$N_{ha}$};
\end{scope}

\begin{scope}[shift={(3,-3)}]
\node[shape=circle,draw] (t1) at (0,0) {};
\node[shape=circle,draw] (t2) at (0.9,-0.7) {};
\node[shape=circle,draw] (t3) at (0.1,-1.5) {};
\node[shape=circle,draw] (t4) at (-0.8,-0.8) {};

\draw[->,dashed] (t1) to [out=-10,in=120]
node[below]{$\scriptstyle\lambda_{ct}$} (t2);
\draw[->,dashed] (t2) to [out=-110,in=20]
node[left]{$\scriptstyle\lambda_{ct}$} (t3);
\draw[->,dashed] (t2) to 
node[above]{$\scriptstyle\lambda_{ct}$} (t4);
\draw[->,dashed] (t3) to [out=180,in=-70]
node[right]{$\scriptstyle\lambda_{ct}$} (t4);
\draw[->,dashed] (t4) to [out=80,in=-170]
node[right]{$\scriptstyle\lambda_{ct}$} (t1);

\node[cloud, cloud puffs=24, draw,minimum width=3.5cm, minimum
height=3.5cm,thin] at (0,-0.7) {};

\node at (1.3,-0.2) {$N_{ct}$};
\end{scope}

\draw[<->,thick,dashed] (1,-1) to node[below]
{$\scriptstyle\delta_n$} (2,-1);

\draw[<->,thick,dashed] (3,-1.8) to node[right]
{$\scriptstyle\delta_{ct}$} (3,-2.8);

\draw[<->,thick,dashed] (0.2,-3.8) to node[below]
{$\scriptstyle\delta_{ct}$} (2,-3.8);

\draw[<->,thick,dashed] (-0.7,-2.7) to node[left]
{$\scriptstyle\delta_{ha}$} (-0.2,-1.7);

\draw[<->,thick,dashed] (0.7,-1.3) to node[above]
{$\scriptstyle\delta_{ct}$} (2.7,-3);

\draw[<->,thick,dashed] (0,-3.5) to node[above]
{$\scriptstyle\delta_{ha}$}(2.8,-1.6);



\end{tikzpicture}
\caption{A schematic sketch of~$\digr(A)$ of Example~\ref{ex:lambdas} (the general case)}
\label{f:incompar}
\end{figure}
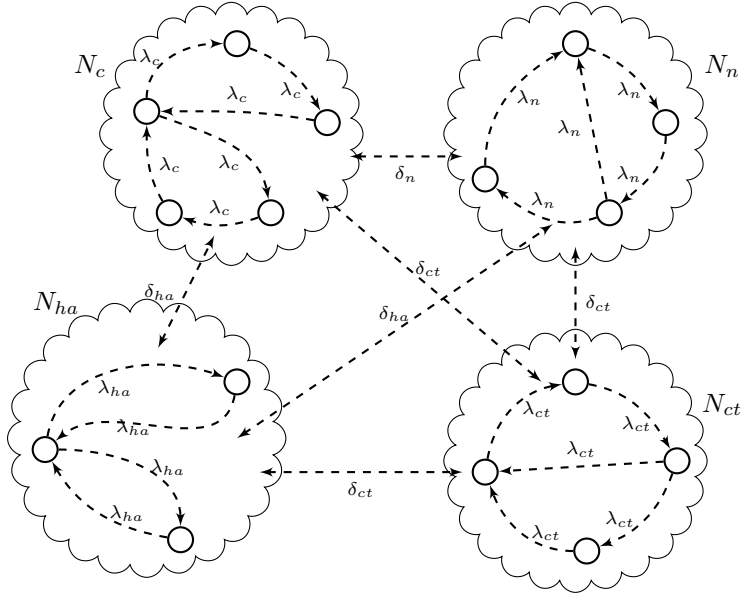

Assume that $\delta_x$ satisfies
\begin{equation}
\label{e:deltax}
\delta_x\leq\min (\lambda_x,\  
\min\{\delta_y\mid\lambda_y>\lambda_x\}).
\end{equation}
Then $\digr(A)$ is also max-balanced (since it can be shown that each edge $(i,j)$ with $i\neq j$ is on a cycle where it has the smallest
weight).

We also see that $\lambda_c=\lambda(A)$, $\lambda_n=\lambda(\bnacht)$, while
$\lambda_{ha}$ and $\lambda_{ct}$ are ``candidates'' for $\lambda(\bharg)$ and
$\lambda(\bct)$, respectively. 
To enforce the correct behaviour of threshold graphs and ensure that $\lambda(\bharg)=\lambda_{ha}$ and
$\lambda(\bct)=\lambda_{ct}$, we set:\\
1) $\delta_n=\lambda_n$,\\
2) $\delta_{ha}=\lambda_{ha}-s$, where $s$ is chosen in such a way that the inequality
$$(\ell(\cycle)-2)\cdot\lambda_n+2(\lambda_{ha}-s)\geq\ell(\cycle)\cdot\lambda_{ha}$$
holds at least for one cycle $\cycle$ containing a node in $N_{ha}$ and a node in
$N_c\cup N_n$;\\
3) $\delta_{ct}$ not greater than $\delta_h$ (for the sake of max-balancing)
and such that the mean weight
of each cycle containing a node of $N_{ct}$
and a node of $N\backslash N_{ct}$ is strictly less than $\lambda_{ct}$.\\
Observe, in particular, that condition 2) ensures that $\thrct(\mu)$ does not gain any
new component as $\mu$ decreases from $\lambda_c$ to $\lambda_{ha}$ so that $\lambda_{ct}<\lambda(\bharg)$,
and that condition 3) ensures $\lambda_{ct}=\lambda(\bct)$. Note that~\eqref{e:gl-matx} satisfies 
conditions 1)-3).


\end{example}

\if{

Consider the following matrix~$A$, described by its
digraph~$\digr(A)$. The node set~$N$ of~$\digr(A)$ is partitioned
into $N=N_{\mathrm C}\cup N_{\mathrm N}\cup N_{\mathrm {HA}} \cup
N_{\mathrm {CT}}$. Further choose parameters $\lambda_\mathrm{C} >
\lambda_{\mathrm N} > \lambda_{\mathrm {HA}} > \lambda_{\mathrm
{CT}}>z$ such that $\lambda_{\mathrm {HA}}\leq (2\lambda_{\mathrm
{CT}} + \lambda_{\mathrm N})/3$. For example, set
$\lambda_\mathrm{C}=0$, $\lambda_{\mathrm N} = -1$,
$\lambda_{\mathrm {HA}} = -3$, $\lambda_{\mathrm {CT}} = -4$,
$z=-5$. The edges in~$\digr(A)$ are as follows: For each
$X\in\{{\mathrm C},{\mathrm N},{\mathrm {HA}},{\mathrm {CT}}\}$ the
nodes in~$N_X$ are connected by one cycle of length~$\lvert
N_X\rvert$ whose edges all have weight $\lambda_X$. Choose an
arbitrary node $i_X$ in $N_X$. There are 6 additional edges
in~$\digr(A)$: The two edges $(i_\mathrm{C},i_\mathrm{N})$ and
$(i_\mathrm{N},i_\mathrm{C})$ with weight $\lambda_\mathrm{N}$, the
two edges $(i_\mathrm{N},i_\mathrm{HA})$ and
$(i_\mathrm{HA},i_\mathrm{C})$ with weight $\lambda_\mathrm{CT}$,
and the two edges $(i_\mathrm{HA},i_\mathrm{CT})$ and
$(i_\mathrm{CT},i_\mathrm{HA})$ with weight $z$.
Figure~\ref{f:incompar} depicts digraph $\digr(A)$.

The cycle cover condition holds in~$\digr(A)$, and so it is
max-balanced. Indeed, the cycle cover condition trivially holds for
all edges contained in a single node set~$N_X$ and for the edges of
the cycles of length~$2$. But it also holds for the two edges
$(i_\mathrm{HA},i_\mathrm{C})$ and $(i_\mathrm{N},i_\mathrm{HA})$
because $\lambda_\mathrm{N} > \lambda_\mathrm{CT}$.

The critical graph of~$A$ is induced by the node set~$N_\mathrm{C}$.
We can check that indeed $\mu^{ha}=\lambda_\mathrm{HA}$ and
$\mu^{ct}=\lambda_\mathrm{CT}$.
This implies that the edge set of~$\bnacht$ is induced by $N_\mathrm{N}\cup
N_\mathrm{HA}\cup N_\mathrm{CT}$, the edge set of~$\bharg$ by
$N_\mathrm{HA}\cup N_\mathrm{CT}$, and the edge set of~$\bct$ by
$N_\mathrm{CT}$.
In particular, $\lambda(\bnacht)=\lambda_\mathrm{N}$,
$\lambda(\bharg)=\lambda_\mathrm{HA}$, and $\lambda(\bct)=\lambda_\mathrm{CT}$.

}\fi

\section{Main Results}\label{s:statements}
In this section, we present the main results of this paper.
These bounds of this section use the following graph parameters of a
digraph~$\digr$:
\begin{itemize}
\item {\em size} $|\digr|$: the number of nodes of~$\digr$,
\item  {\em circumference} $\circumf(\digr)$: the greatest length of a cycle in graph $\digr$,
\item  {\em cab driver's diameter} $\cabdrive(\digr)$: the greatest length of a path in $\digr$,
\item {\em max-girth} $\hat{\g}(\subcrit)$: the greatest {\em girth}, i.e.,
shortest cycle length, of strongly connected components of~$\digr$
\item {\em max-cyclicity} $\hat{\gamma}(\subcrit)$: the greatest cyclicity of
strongly connected components of~$\digr$.
\end{itemize}

The computation of the circumference~$\circumf(\digr)$ and the cab
driver's diameter~$\cabdrive(\digr)$ are both NP-hard in the number of nodes
of~$\digr$.
However, they can be upper bounded by~$\lvert\digr\rvert$
and~$\lvert\digr\rvert-1$, respectively.

Denote by~$\lVert A \rVert$ the difference between the largest and
the smallest finite (i.e., $\neq\bzero$) entry of~$A$, and by~$n_B$
the size of the smallest submatrix of~$B$ containing all its finite
entries.

We explained in the introduction that $T(A)\le\max\big(T_1(A,B),T_2(A,B)\big)$.
Our main results are bounds on~$T_1$ and~$T_2$.
All of them are mutually incomparable.

\begin{theorem}\label{t:T1ha}
For any matrix $A\in\Rpnn$,  if $B=\bnacht$ or $B=\bharg$, we have the following bounds
\begin{align}
T_1(A,B)&\leq \wiel(n)\label{e:T1haWielandt}\\
T_1(A,B)&\leq \hat{\g}(n -2)+n\label{e:T1haDM}\\
T_1(A,B)&\leq (\hat{\g}-1)(\circumf-1)+(\hat{\g}+1)\cabdrive\label{e:T1haCB}
\end{align}
where~$\hat{\gamma}=\hat{\gamma}(\crit(A))$,
$\hat{\g}=\hat{\g}(\crit(A))$, $\circumf=\circumf(\digr(A))$,
$\cabdrive=\cabdrive(\digr(A))$.
\begin{align}
 T_1(A,B)&\leq \hat{\gamma}(n-2)+n-n_c+\ep(\crit(A))\label{e:T1haDMEp}\\
 T_1(A,B)&\leq (\hat{\gamma}-1)(\circumf-1)+(\hat{\gamma}+1)\cabdrive+\ep(\crit(A))\label{e:T1haCBEp}
\end{align}
where $n_c=|\crit(A)|$ (i.e., the number of critical nodes) and
$\ep(\crit(A))$ is the exploration penalty of
$\crit(A)$ (see Definition~\ref{def:TcrEp}).
\end{theorem}

The exploration penalty $\ep(\crit(A))$ is a quantity that depends
only on the critical graph and can be bounded by its index, see
Section~\ref{s:Ep} for further details.

\begin{remark}\label{r:Boolean2}
As we noted in Remark~\ref{r:Boolean}, those bounds apply to the
transient of Boolean matrices. We thus recover the bound of
Wielandt~\cite{Wie-50} in~\eqref{e:T1haWielandt} and the bound
of Dulmage and Mendelsohn~\cite{DM-64} in~\eqref{e:T1haDM}. Notice that
\eqref{e:T1haDM}
implies~\eqref{e:T1haWielandt} if~$\hat{\g}\le n-1$. The remaining
case is trivial for Boolean matrices because there is only one such
matrix, but in the non-Boolean case we need a different strategy.
(Proposition~\ref{p:TcrHAWielandt} below).

Bound~\eqref{e:T1haWielandt} is optimal in the sense that the bound
is reached for any~$n$, as was already noted in~\cite{Wie-50},
while bound~\eqref{e:T1haDM} is reached if and only if $\hat{\g}$
and~$n$ are coprime (see~\cite{S-85}).
\end{remark}


%
%

Iterating the process of weak CSR expansion, we get the following improvement of~\cite{SS-11}[Theorem~4.2]:
\begin{corollary}[CSR decomposition]\label{c:T1ha}
For any matrix $A\in\Rpnn$, there are some matrices $C_i,S_i,R_i$ defined by induction
with~$S_i$ diagonally similar to Boolean periodic matrices and some
scalars $\lambda_i\in\R$, where $i$~varies between~$1$ and~$K\le n$ such
that we have:
$$\forall t\ge \min\left(\wiel(n),(n-2)\circumf(\digr(A))+n\right), A^t=\bigoplus_{k=1}^K\lambda_i^{\otimes t}C_iS_i^tR_i.$$
\end{corollary}

\begin{theorem}\label{t:T1ct}
For any matrix $A\in\Rpnn$, we have the following bounds
\begin{align}
T_1(A,\bct)&\leq \wiel(n)\label{e:T1ctW}\\
T_1(A,\bct)&\leq (n-1)\circumf+\min(n,\cabdrive +\circumf+1)\label{e:T1ctCD}\\
T_1(A,\bct)&\leq (\cabdrive+\circumf-1)\circumf +\cabdrive+1\label{e:T1ctCB}
\end{align}
where $\circumf=\circumf(\digr(A))$ and $\cabdrive=\cabdrive(\digr(A))$.
\end{theorem}

The proof of those theorems is explained in Section~\ref{s:ProofS}
and performed in Sections~\ref{s:CsrtoWalk} and~\ref{s:TcrToT1}. Now
we also state bounds on~$T_2(A,B)$ and on~$T_2(A,B,v)$.

\begin{theorem}\label{t:T2}
 Let $A\in\Rpnn$ be irreducible and let~$B$ be
subordinate to~$A$. Denote $\cabdrive_B:=\cabdrive(\digr(B))$ and
$\hat{\gamma}=\hat{\gamma}(\crit(A))$.\\
If $\lambda(B)=\bzero$, then $T_2(A,B)\le \cabdrive_B+1\le n_B$.
Otherwise, we have the following bounds
\begin{equation}
\label{e:T2a}
\begin{split}
T_2(A,B)&\le \frac{(n^2-n+1)(\lambda(A) -\min_{ij}a_{ij})
+\cabdrive_B (\max_{ij}b_{ij}-\lambda(B))}{\lambda(A)-\lambda(B)}\\
&\le\frac{n^2-n+1}{\lambda(A)-\lambda(B)}\|A\|+\cabdrive_B
\end{split}
\end{equation}
\begin{equation}
\label{e:T2b}
\begin{split}
T_2(A,B) &\le \frac{\left(\hat{\gamma} (n-1)+n\right)(\lambda(A)
-\min_{ij}a_{ij})
+\cabdrive_B(\max_{ij}b_{ij}-\lambda(B))}{\lambda(A)-\lambda(B)}\\
&\le\frac{\hat{\gamma}(n-1)+n}{\lambda(A)-\lambda(B)}\|A\|+\cabdrive_B
\end{split}
\end{equation}
\begin{equation}
\label{e:T2c}
\begin{split}
T_2(A,B)&\le
\frac{\left((\hat{\gamma}-1)\circumf+(\hat{\gamma}+1)\cabdrive\right)(\lambda(A)
-\min_{ij}a_{ij})
+\cabdrive_B(\max_{ij}b_{ij}-\lambda(B))}{\lambda(A)-\lambda(B)}\\
&\le\frac{(\hat{\gamma}-1)\circumf+(\hat{\gamma}+1)\cabdrive}{\lambda(A)-\lambda(B)}\|A\|+\cabdrive_B.
\end{split}
\end{equation}
If~$A$ has only finite entries, then we have:
\begin{align}
\label{e:T2fin}
T_2(A,B)&\le \frac{2(\lambda(A)-\min_{ij}a_{ij})
+(\lambda(B)-\min_{ij}b_{ij})}{\lambda(A)-\lambda(B)}
\le\frac{3\|A\|}{\lambda(A)-\lambda(B)}\\
\label{e:T2finSyK}
T_2(A,B)&\le \frac{2(\lambda(A)-\min_{ij}a_{ij})}{\lambda(A)-\lambda(B)}+\cabdrive_B
\le\frac{2\|A\|}{\lambda(A)-\lambda(B)}+\cabdrive_B.
\end{align}
\end{theorem}

The following theorem generalizes Proposition~5 of~\cite{CBFN-12} and Theorem~3.5.12 of~\cite{SyK:03}.

\begin{theorem}
\label{t:T2v} Let $A\in\Rpnn$ be irreducible, $B$ be
subordinate to~$A$ and $v$ be a vector with only finite entries, i.e.,
$v\in\mathds{R}^n$.

If $\lambda(B)=\bzero$, then $T_2(A,B,v)\le T_2(A,B)\le
\cabdrive_B+1\le n_B$. Otherwise, we have the following bound:
\begin{equation}
\label{e:T2v}
 T_2(A,B,v)\leq \frac{\lVert v\rVert +(n-1)\lVert
A\rVert}{\lambda(A)-\lambda(B)}
\end{equation}

If~$A$ has only finite entries, then we have:
\begin{equation}
\label{e:T2fin:vec}
T_2(A,B,v)\leq \frac{\lVert v\rVert+(\lambda(A)-\min_{ij}a_{ij})
+(\lambda(B)-\min_{ij}b_{ij})}{\lambda(A)-\lambda(B)}
\le\frac{2\|A\|+\|v\|}{\lambda(A)-\lambda(B)}.
\end{equation}
\end{theorem}

The proofs of Theorems~\ref{t:T2} and~\ref{t:T2v} are deferred to Section~\ref{s:TcrToT2}.

\begin{remark}\label{r:complexityInN}
The bounds on $T_1$ are quadratic in~$n$, but even if one fixes the size of the entries (for instance entries are $-\infty,0$ or~$1$),
the general bounds on~$T_2$ have degree~$4$, because $\frac{1}{\lambda(A)-\lambda(B)}$ can be as big as $\|A\| (n^2-1)/4$.
(Take two cycles with length $(n+1)/2$ and $(n-1)/2$ which both have weight~$1$).

For the same reason, both bounds are quadratic if all entries are finite.
\end{remark}

\section{Comparison to Previous Transience Bounds}\label{s:comparison}

Hartmann and Arguelles~\cite{HA-99} proved one transience bound for irreducible
max-plus matrices and one for irreducible max-plus systems with finite initial
vector.
These two  bounds are, respectively,
\[
\max\left( 2n^2 \ ,\ \frac{2n^2}{\lambda(A) - \lambda(\bharg)} \lVert A\rVert
\right)
\quad
\text{and}
\quad
\max\left( 2n^2 \ ,\ \frac{\lVert v\rVert + n\lVert A\rVert}{\lambda(A) -
\lambda(\bharg)}
\right)
\enspace.
\]
Combination of our bounds in~\eqref{e:T1ctW} and~\eqref{e:T2a},
respectively~\eqref{e:T2v}, yields
bounds that are strictly lower than that of Hartmann and Arguelles.
Note that our results, being more detailed, allow a considerably more
fine-grained
analysis of the transient phase.
For instance,
there exist matrices for which $\lambda(\bct)=\bzero$ but
$\lambda(\bharg)\neq \bzero$
(cf.\ Example~\ref{ex:lambdas}).
Our bounds show in particular that the transients of these matrices and systems are at most
$\wiel(n)$,
which cannot be deduced from previous bounds, including that of Hartmann
and
Arguelles.

Bouillard and Gaujal~\cite{BG-00} and Akian et al.~\cite{AGW-05}
gave transience bounds for irreducible matrices in the case that the
cyclicity of the critical graph is equal to~$1$. They explained how
to extend their bounds to arbitrary cyclicities, but that reduction
involves multiplying the bound by the cyclicity of the critical
graph or its subgraph. Akian et al.~\cite{AGW-05} derive bounds for
the periodicity transient of $\{a_{ij}^{(t)}\}_{t\geq 1}$ for fixed
$i,j$ instead of the whole matrix powers, and show that their
bounding techniques extend to the case of matrices of infinite
dimensions. We discuss the relation of this approach to weak CSR
expansions in more detail in Section~\ref{s:localred}.

Soto y Koelemeijer~\cite{SyK:03} (Theorem 3.5.12) established a transience bound for matrices
whose entries are all finite.
In our notation, it reads
\[
\max\left( 2n^2 \ ,\ \frac{2\lVert A\rVert}{\lambda(A)-\lambda(\bnacht)} + n+1
\right)
\enspace.
\]
Combination of our bounds in~\eqref{e:T1ctW} and~\eqref{e:T2finSyK}
yields a bound of\\ $\max\big( \wiel(n) \,,\, 2\lVert A\rVert /
(\lambda(A) - \lambda(\bct))+\cabdrive_B\big)$, which is strictly lower.
In many cases, it is even better to use~\eqref{e:T2fin}.

Charron-Bost et al.~\cite{CBFN-12} gave two transience bounds for
systems. They also explained how to transform transience bounds for
systems into transience bounds for matrices. Combination of our
bounds~\eqref{e:T1haCB}, \eqref{e:T1haCBEp}, and \eqref{e:T2v}
yields bounds that are strictly lower than those of~\cite{CBFN-12}.

\section{Proof Strategy}\label{s:ProofS}

In this section, we outline the proof of the bounds on~$T_1$ stated in Theorems~\ref{t:T1ha} and~\ref{t:T1ct}.
Moreover, we provide some general statements that can be used to get a better
bound if more information on the matrix is available.

In all proofs, we assume $\lambda(A)=\bunity$ (replacing~$A$ by
$\lambda(A)^- \otimes A$ if necessary).

The first stage of the proof is the  following representation theorem for $CS^tR$ expansions.
\begin{theorem}[CSR and walks]\label{t:representation}
Let~$A\in\Rpnn$ be a matrix with~$\lambda(A)=\bunity$ and $C,S,R$ be
the CSR terms of~$A$ with respect to some completely reducible
subgraph~$\subcrit$ of the critical graph~$\crit(A)$.

Let~$\gamma$ be a multiple of~$\gamma(\subcrit)$ and~$\mN$ a set of
critical nodes that contains one node of every s.c.c.\ of~$\subcrit$.

Then we have, for any~$i,j$ and~$t\in\Nat$:
\begin{equation}\label{e:representation}
(CS^tR)_{ij}=p\left(\walkslennode{i}{j}{t,\gamma}{\mN}\right)
\end{equation}
where $\walkslennode{i}{j}{t,\gamma}{\mN}:=\left\{W\in\walksnode{i}{j}{\mN}
\,\big|\,l(W)\equiv t\pmod\gamma\right\}$
\end{theorem}

The proof of this theorem is deferred to Section~\ref{s:CsrtoWalk}.

Observe that it implies Proposition~\ref{p:purely} as well as the
following corollary.

\begin{corollary}\label{c:csr-indep}
$CS^tR$ depends only on the set of s.c.c.'s of~$\crit(A)$
intersecting with~$\subcrit$.
\end{corollary}

Let $\subcrit_1,\ldots,\subcrit_l$ be the s.c.c. of $\crit(A)$ with
node sets $N_1,\ldots N_l$, and let $C_{\subcrit_1}$,
$S_{\subcrit_1}$, $R_{\subcrit_1}$ be the CSR terms defined with
respect to $\subcrit_1$.  For $\nu=2,\ldots,l$, we define a
subordinate matrix $A^{(\nu)}$ by setting the entries of $A$ with
rows and columns in $N_1\cup\ldots\cup N_{\nu-1}$ to $\0$, and let
$C_{\subcrit_\nu}$, $S_{\subcrit_\nu}$, $R_{\subcrit_\nu}$ be the
CSR terms defined with respect to $\subcrit_\nu$ in $A^{(\nu)}$.

\begin{corollary}\label{c:representation}
If $\subcrit_1,\cdots,\subcrit_l$ are the s.c.c.'s of~$\crit(A)$,
then we have:
\begin{equation}\label{e:early-exp}
C S^t R=\bigoplus_{\nu=1}^lC_{\subcrit_\nu} S_{\subcrit_\nu}^t R_{\subcrit_\nu}.
\end{equation}
\end{corollary}
\begin{proof}
Using Theorem~\ref{t:representation}, observe that the set of walks
$\walkslennode{i}{j}{t,\gamma}{\crit(A)}$, where $\gamma$ is the
cyclicity of $\crit(A)$, can be decomposed into the sets
$\mathcal{W}_{\nu}$ consisting of walks in
$\walkslennode{i}{j}{t,\gamma}{\subcrit_\nu}$ that do not visit any
node of $\subcrit_1,\ldots,\subcrit_{\nu-1}$, for $\nu=1,\ldots l$
(in particular,
$\mathcal{W}_1=\walkslennode{i}{j}{t,\gamma}{\subcrit_1}$).
\end{proof}

Corollary~\ref{c:representation}, which will be useful in the final
section of the paper, and Corollary~\ref{c:T1ha} are different examples of the CSR decomposition schemes considered by
Sergeev and Schneider~\cite{SS-11}.

If $\digr$ is the graph defining~$B$ in~\eqref{e:CSRschemes}, it contains~$\crit(A)$ and by the optimal walk interpretation~\eqref{e:walksense1}, we have:
\begin{equation*}
\begin{split}
a^{(t)}_{ij}&=b^{(t)}_{ij}\oplus
p\left(\walkslennode{i}{j}{t}{\digr}\right)\\
&= b^{(t)}_{ij}\oplus
p\left(\walkslennode{i}{j}{t}{\crit(A)}\right)\oplus
p\left(\walkslennode{i}{j}{t}{\digr}\setminus
\walkslennode{i}{j}{t}{\crit(A)}\right).
\end{split}
\end{equation*}

The proof that~$T_1(A,B)\le T$ has two parts:
\begin{enumerate}
 \item\label{SdP}
 Scheme-dependent part: show that for~$t\ge T$ we have
 \begin{equation}\label{e:SdP}
p\left(\walkslennode{i}{j}{t}{\digr}\setminus \walkslennode{i}{j}{t}{\crit(A)}\right)\le p\left(\walkslennode{i}{j}{t,\gamma}{\crit(A)}\right).
 \end{equation}
 \item\label{SiP}
 Scheme-independent part: show that for~$t\ge T$ we have
  \begin{equation}\label{e:SiP}
  p\left(\walkslennode{i}{j}{t,\gamma}{\crit(A)}\right)\le p\left(\walkslennode{i}{j}{t}{\crit(A)}\right).
 \end{equation}
\end{enumerate}

By Theorem~\ref{t:representation}, we have $\displaystyle
p\left(\walkslennode{i}{j}{t}{\crit(A)}\right)\le
p\left(\walkslennode{i}{j}{t,\gamma}{\crit(A)}\right)=(CS^tR)_{ij}.$
Thus, \eqref{e:SdP} implies $A^t\le B^t\oplus CS^tR$, while
\eqref{e:SiP} implies~$A^t\ge B^t\oplus CS^tR$.

Let us go deeper into the strategy for each part.

\begin{enumerate}
 \item The scheme-dependent part goes as follows
 \begin{enumerate}
  \item For~$B=\bnacht$, $\digr=\crit(A)$ and there is nothing to prove.
  \item For~$B=\bharg$, we take a walk~$W$ with maximal weight
  in~$\walkslennode{i}{j}{t}{\digr}\setminus
  \walkslennode{i}{j}{t}{\crit(A)}$ and a closed walk~$V$
  from a node of~$W$ to~$\crit(A)$ and back whose edges have
  weight greater than or equal to the greatest weight of the
  edges in~$W$. Then, we insert $V^{\gamma(\crit(A))}$
  (i.e., $\gamma(\crit(A))$ copies of $V$) in~$W$, and
  remove as many cycles of the new walk as possible,
  preserving the length modulo $\gamma(\crit(A))$ until we
  get a walk~$\tilde{W}$ with length at most~$t$.

 As a result, we thus replaced some edges of~$W$ by edges
 with greater weight and removed other edges, so
 $p(\tilde{W})\ge p(W)$.


  \item For~$B=\bct$, we also take a walk~$W$ with maximal weight in~$\walkslennode{i}{j}{t}{\digr}\setminus \walkslennode{i}{j}{t}{\crit(A)}$
but now we replace some cycles of~$W$ by some copies of a
cycle with greater mean weight, to get a new walk with
length~$t$. We therefore introduce the concept of a ``staircase'' of cycles,
and Lemma~\ref{l:staircase} will ensure us that we
can iterate this process and eventually reach a critical
node.

Note that we need to remove cycles before we replace them and to have some steps with non-critical cycles, which explains why the bound for~$T_1(A,\bct)$ are larger than the one for~$T_1(A,\bharg)$ and~$T_1(A,\bnacht)$.
However, the worst case remains~$\wiel(n)$.
\end{enumerate}

 \item
By Theorem~\ref{t:representation}, to have~\eqref{e:SiP}, it is enough to prove
that for each s.c.c.~$H$ of~$\crit(A)$ there is a~$\gamma\in\Nat$ and a set
of nodes~$\mN\subset H$ such that
\begin{equation}\label{e:WalkEq}
p\left(\walkslennode{i}{j}{t}{H}\right)\ge p\left(\walkslennode{i}{j}{t,\gamma}{\mN}\right).
\end{equation}

To ensure that Equation~\eqref{e:WalkEq} is satisfied, we use the following steps:
\begin{enumerate}
 \item\label{step1} For each s.c.c.~$H$ of~$\crit(A)$, choose
$\mN\subset H$ and $\gamma$ a multiple of~$\gamma(H)$
 and take a walk~$W$ such that $p(W)=p(\walkslennode{i}{j}{t,\gamma}{\mN})$.
 \item\label{step2} Remove as many cycles as possible from~$W$, keeping it in~$\walkslennode{i}{j}{t,\gamma}{\mN}$.
 \item\label{step3} Insert critical cycles so that the new walk has length~$t$.
\end{enumerate}
Since~$\lambda(A)\le \1$, steps~\ref{step2} and~\ref{step3}
cannot strictly increase the weight of the walk,
so~\eqref{e:WalkEq} is satisfied.

\end{enumerate}

It is clear from the strategy that the main point is to remove closed walks from a given walk, while preserving the length modulo some given integer.
This will be the subject of Section~\ref{s:Tcr}. We will use three different
tactics, one of them is completely new.
Different bounds depending on different parameters arise from
different choices of~$\mN$ and~$\gamma$ in step~\ref{step1} and different
tactics in step~\ref{step2}.
To reach the (optimal) Wielandt number~$\wiel(n)$, we have to combine two of
them.

To state general results, we introduce two graph-theoretic
quantities.
\begin{definition}\label{def:TcrEp}
Let~$\digr$ be a subgraph of~$\digr(A)$ and~$\gamma\in\Nat$.
\begin{enumerate}
 \item The {\em cycle removal threshold}~$T_{cr}^\gamma(\subcrit)$,
(resp.\ the {\em strict cycle removal
threshold~$\tilde{T}^\gamma_{cr}(\subcrit)$}) of~$\subcrit$ is
the smallest nonnegative integer~$T$ for which the following
holds: for all walks~$W\in\walksnode{i}{j}{\subcrit}$ with
length~$\geq T$, there is a walk
$V\in\walksnode{i}{j}{\subcrit}$ obtained from~$W$ by removing
cycles (resp.\ at least one cycle) and possible inserting cycles
of~$\subcrit$ such that $l(V)\equiv l(W) \pmod{\gamma}$ and
$l(V)\le T$.
 \item The {\em exploration penalty}~$\ep^\gamma(i)$ of a node~$i\in\crit(A)$
is the least~$T\in\Nat$ such that for any multiple~$t$
of~$\gamma$ greater or equal to~$T$, there is a closed walk
on~$\crit(A)$ with length~$t$ starting at~$i$.

The {\em exploration penalty}~$\ep^\gamma(\subcrit)$
of~$\subcrit\subseteq\crit(A)$ is the maximum of
the~$\ep^\gamma(i)$ for~$i\in\subcrit$. We further set
$\ep((\crit(A))=\max_l\ep^{\gamma(\crit_l)}(\crit_l)$, which is
the quantity used in Theorem~\ref{t:T1ha}.
\end{enumerate}
\end{definition}

Obviously,
$T^\gamma_{cr}(\subcrit)\le\tilde{T}^\gamma_{cr}(\subcrit)\le
T^\gamma_{cr}(\subcrit)+1$ but it will be useful to have both
definitions.

Bounds on~$\ep^\gamma$ are given in Section~\ref{s:Ep} while~$T^\gamma_{cr}$ is investigated in Section~\ref{s:Tcr}
We can already notice the following.

First, $\ep^\gamma(i)$ is finite if and only if $\gamma$ is a
multiple of the cyclicity of its s.c.c.\ in~$\crit(A)$. Second,
if~$\gamma$ is multiplied by an integer, then $\ep^\gamma(i)$
decreases  but $T_{cr}^\gamma(\subcrit)$ increases. Third, for
fixed~$\gamma$, $\ep^\gamma(\subcrit)$ decreases when~$\subcrit$
increases. Finally, $\ep^\gamma(i)=0$ if and only if there is a
critical closed walk with length~$\gamma$ at~$i$. Especially, for
any cycle~$\cycle$, we have
$$\ep^{l(\cycle)}(\cycle)=0.$$

This gives two extremal choices for~$\subcrit$ and~$\gamma$:
either $\subcrit$ is a s.c.c.\ of~$\crit(A)$ and~$\gamma$ is its cyclicity,
or $\subcrit$ is a critical cycle and $\gamma$~is its length.

The first choice is used in~\cite{BG-00}, the second one
in~\cite{HA-99} and both choices in~\cite{CBFN-12}. Here we
systematically test those two choices. The first one is used to
prove the bounds in Theorem~\ref{t:T1ha} that depend on
$\ep(\crit(A))$. The second one is used for the other bounds
on~$T_1$.

If other choices prove to be useful under additional assumptions on~$\digr(A)$,
one can apply Proposition~\ref{p:TcrToT1} with other parameters.

The strategy explained in this section leads to the following
proposition, which implies Theorems~\ref{t:T1ha} and~\ref{t:T1ct},
except for~\eqref{e:T1ctW}.


\begin{proposition}[From cycle removal to Weak CSR]\label{p:TcrToT1}
Let $A$ be a square matrix  and $\subcrit$ be a representing
subgraph of~$\crit(A)$ with s.c.c.'s $\subcrit_1,\cdots,\subcrit_m$,
and let $\gamma_l$ be multiples of $\gamma(\subcrit_l)$.
\begin{itemize}
\item[{\rm (i)}] If $B=\bnacht$ or~$B=\bharg$, then $T_1(A,B)\le \max_l
\left(T_{cr}^{\gamma_l}(\subcrit_l)-\gamma_l+1+\ep^{\gamma_l}(\subcrit_l)\right).$

\item[{\rm (ii)}] If $B=\bct$, then $T_1(A,B)\le
\max\left\{\tilde{T}^{l(\cycle)}_{cr}(\cycle)\mid \cycle\textnormal{ cycle in
$\ctgr$}\right\}$
\end{itemize}
\end{proposition}


This proposition is proved in Section~\ref{s:TcrToT1}. 
The bounds of Theorem~\ref{t:T1ha} with~$\hat{\gamma}(\crit(A))$
and~$\ep(\crit(A)$ can be improved if one knows more of the
structure
of~$\crit(A)$.

\section{Proof of Theorem~\ref{t:representation}}\label{s:CsrtoWalk}
Let~$A,\subcrit,\mN, \gamma, t, i, j$
be as in the statement of Theorem~\ref{t:representation}.

We first prove:
\begin{equation}\label{e:representation1}
 \left(CS^tR\right)_{ij}\le p\left(\walkslennode{i}{j}{t,\gamma}{\mN}\right).
\end{equation}
By definition of~$C$, $S$ and~$R$, there are walks $W_1$, $W_2$ and $W_3$ such that $\left(CS^tR\right)_{ij}=p(W_1W_2W_3)$ and
\begin{equation}\label{e:WalkCSR}
l(W_1)\equiv l(W_3)\equiv 0\pmod{\gamma(\subcrit)} \textnormal{, }W_2\subset \subcrit\textnormal{ and } l(W_2)=t.
\end{equation}

Let~$k$ be the start node of~$W_2$.
By hypotheses, $k$ is critical and there is a node~$l$ of~$\mN$ in the same s.c.c.~$H$ of~$\crit(A)$ as~$k$.
Thus there are walks~$W_4$ and $W_5$ with only critical edges, going from~$k$ to~$l$ and from~$l$ to~$k$ respectively.
Thus, $W_4W_5$ is a circuit of~$\crit(A)$ and $p(W_4)+p(W_5)=0$.

Let~$G$ be the s.c.c.\ of~$k$ in~$\subcrit$. As $G\subseteq H$,
$\gamma(H)$ divides $\gamma(G)$, thus also $\gamma(\subcrit)$ and
$\gamma$. Hence $\gamma(H)$ divides $l(W_1)$ and $l(W_3)$. It also
divides~$l(W_4W_5)$ and we have 
$$L=l(W_1)+l(W_3) +l(W_4)+l(W_5)\equiv 0 \pmod{\gamma(H)}.$$
Therefore, for $m\in\Nat$ large enough, there is a closed walk~$W_6$ on~$H$ starting at~$k$ with length $m\gamma -L$.

Set $W=W_1W_4W_6W_5W_2W_3$. By construction $W\in \walkslennode{i}{j}{t,\gamma}{\mN}$
and $p(W)=p(W_1W_2W_3)=\left(CS^tR\right)_{ij}$,
so~\eqref{e:representation1} is proved.

It remains to show:
\begin{equation}\label{e:representation2}
 \left(CS^tR\right)_{ij}\ge p\left(\walkslennode{i}{j}{t,\gamma}{\mN}\right).
\end{equation}

By definition of~$\walkslennode{i}{j}{t,\gamma}{\mN}$ there are a node $l\in\mN$ and two walks~$V_1$ and~$V_2$
going from~$i$ to~$l$ and from~$l$ to~$j$ respectively such that $l(V_1V_2)\equiv t \pmod{\gamma}$ and
$p(V_1)+p(V_2)=p\left(\walkslennode{i}{j}{t,\gamma}{\mN}\right)$.

Let $k$ be a node of~$\subcrit$ in the same s.c.c~$H$ of~$\crit(A)$
as~$l$. As above, there are critical walks~$W_4$ and $W_5$, going
from~$k$ to~$l$ and from~$l$ to~$k$ respectively and $\gamma(H)$
divides~$\gamma$.

Let~$V_3$ be a closed walk in~$\subcrit$ with start node~$k$, whose length is $\ge t+\gamma$.
Let~$V_4$ be its shortest prefix such that 
$l(V_1)+l(W_5)+l(V_4)\equiv 0 \pmod{\gamma}$ and $V_5$ be the
complementary (i.e., $V_3=V_4V_5$). Let $W_2$ be the prefix
of~length~$t$ of~$V_5$ and $V_6$ be its complementary ($V_5=W_2V_6$,
$V_3=V_4W_2V_6$).

Set $W_1=V_1W_5V_4$ and $W_3=V_6V_3^{(\gamma-1)}(W_4W_5)^{(\gamma-1)}W_4V_2$.
By construction $W_1,W_2$ satisfy~\eqref{e:WalkCSR}.
Moreover, we have
$$W_1W_2W_3=V_1W_5V_4W_2V_6V_3^{(\gamma-1)}(W_4W_5)^{(\gamma-1)}W_4V_2=V_1W_5V_3^{\gamma}(W_4W_5)^{(\gamma-1)}W_4V_2$$
so 
$l(W_1W_2W_3)\equiv l(V_1)+l(V_2)\equiv 0\pmod \gamma$ and~$W_3$
also satisfies~\eqref{e:WalkCSR}.

On the other hand $W_5V_3^{\gamma}(W_4W_5)^{(\gamma-1)}W_4$ is a critical closed walk, so it has weight~$0$ and
$p(W_1W_2W_3)=p(V_1)+p(V_2)=p\left(\walkslennode{i}{j}{t,\gamma}{\mN}\right)$,
so~\eqref{e:representation2} is proved.


\section{Proof of Proposition~\ref{p:TcrToT1}}\label{s:TcrToT1}
In this section, we prove Proposition~\ref{p:TcrToT1}, following the
strategy described in Section~\ref{s:ProofS}.

\subsection{Scheme independent part}

In this section, we prove the following lemma.

\begin{lemma}[Scheme independent part]\label{l:SiP}
Let $A$ be a square matrix with $\lambda(A)=0$ and $\subcrit$ be a
representing subgraph of~$\crit(A)$ with s.c.c.\
$\subcrit_1,\ldots,\subcrit_m$ and $\gamma_l$ be multiples
of~$\gamma(\subcrit_l)$, for $l=1,\ldots, m$.

For any $t\ge \max_l
\left(T_{cr}^{\gamma_l}(\subcrit_l)-\gamma_l+1+\ep^{\gamma_l}(\subcrit_l)\right)$
and any~$i,j$, inequality~\eqref{e:SiP}, with $\subcrit$ instead of
$\crit(A)$, holds for~$\gamma=\operatorname{lcm}_l \gamma_l$.
\end{lemma}
\begin{proof}
Indeed, any walk $W\in\walkslennode{i}{j}{t,\gamma}{\subcrit}$ is in
$\walkslennode{i}{j}{t,\gamma_l}{\subcrit_l}$ for some~$l$. By
definition of~$T_{cr}^{\gamma_l}(\subcrit_l)$, there is a
walk~$V\in\walkslennode{i}{j}{t,\gamma_l}{\subcrit_l}$ with length
at most~$T_{cr}^{\gamma_l}(\subcrit_l)$ and $p(V)\ge p(W)$.

If $t\ge
T_{cr}^{\gamma_l}(\subcrit_l)-\gamma_l+1+\ep^{\gamma_l}(\subcrit_l)$,
then $t-l(V)\ge\ep^{\gamma_l}(\subcrit_l)-\gamma_l+1$. Since
$t-l(V)$ and $\ep^{\gamma_l}(\subcrit_l)$ are both multiples
of~$\gamma_l$, it implies $t-l(V)\ge\ep^{\gamma_l}(\subcrit_l)$, so
there is a closed walk on~$\crit(A)$ with length~$t-l(V)$ at each
node of~$\subcrit_l$. Inserting such a walk in~$V$ where it
reaches~$\subcrit_l$, we get a new walk~$\tilde{W}\in
\walkslennode{i}{j}{t}{\subcrit_l}\subseteq
\walkslennode{i}{j}{t}{\subcrit}$ with $p(\tilde{W})=p(V)\ge p(W)$.
\end{proof}

\subsection{Hartmann and Arguelles scheme}

In this section, we perform step~\ref{SdP} of the strategy in the case
$B=\bharg$.
We prove the following lemma.

\begin{lemma}\label{l:SdPHA}
Let $A$ be a square matrix with $\lambda(A)=0$ and $\subcrit$ be a
representing subgraph of~$\crit(A)$ with s.c.c.\
$\subcrit_1,\cdots,\subcrit_m$ and $\gamma_l$ be multiples
of~$\gamma(\subcrit_l)$ for $l=1,\ldots,m$.

For any $t\ge \max_l
\left(T_{cr}^{\gamma_l}(\subcrit_l)-\gamma_l+1\right)$ and
any~$i,j$, inequality~\eqref{e:SdP} holds
with~$\gamma=\operatorname{lcm}_l \gamma_l$ and~$\digr=\hagr$ (the
graph defining~$\bharg$ in Section~\ref{s:csr}).
\end{lemma}
\begin{proof}
We assume without loss of generality that~$A$ is max-balanced.

Let~$W$ be a walk with maximal weight
in~$\walkslennode{i}{j}{t}{\digr}\setminus
\walkslennode{i}{j}{t}{\crit(A)}$ We show that there exists a walk
$\tilde{W}\in\walkslennode{i}{j}{t,\gamma}{\crit(A)}$
with~$p(\tilde{W})\ge p(W)$.

Denote the maximum weight of edges in~$W$
by $\mu(W)$. Define the graph
\begin{equation}
\label{tildedigr}
\tilde{\digr}:=
\begin{cases}
\hagr &\text{if $\mu(W)\le\mu^{ha}$}\enspace,\\
\thrha(\mu(W)) & \text{otherwise}\enspace.
\end{cases}
\end{equation}
By the definition of Hartmann-Arguelles threshold graphs,
$\crit(A)\subseteq \Tilde{\digr}\subseteq\hagr$.  In both cases
of~\eqref{tildedigr}, walk $W$ contains a node~$k$ of
digraph~$\tilde{\digr}$, which is completely reducible (due to the
max-balancing).

Let $W=W_1\cdot W_2$ with~$W_1$ ending at node~$k$. By definition
of~$\Tilde{\digr}$, there exists a critical node~$\ell$ in the same
s.c.c.~$H$ of~$\tilde{\digr}$ as~$k$.
Moreover, $H$ contains a whole s.c.c.\ $\crit_l(A)$ of~$\crit(A)$, and hence 
also a component $\subcrit_l$ of the representing subgraph $\subcrit$.
Hence we can choose $\ell$~in $\subcrit_l$. 

Let~$V_1$ be a walk in~$\tilde{\digr}$ from~$k$ to~$\ell$ and
$V_2$~be a walk in~$\tilde{\digr}$ from~$\ell$ to~$k$. Set
$V=V_1V_2$
and~$W_3 = W_1\cdot V^{\gamma_l}\cdot W_2$.
 By the definition of
the cycle replacement threshold, there exists a
walk~$\tilde{W}\in\walkslennode{i}{j}{t,\gamma_l}{\subcrit_l}$
obtained from~$W_3$ by removing cycles and possibly inserting cycles
in~$\subcrit_l$ such
that~$l(\tilde{W})\le T_{cr}^{\gamma_l}(\subcrit_l)\le t+\gamma_l
-1$. Since $l(W_3)\equiv t\pmod{\gamma_l}$, it
implies~$l(\tilde{W})\le t$.

Recall that since~$A$ is max-balanced and~$\lambda(A)=0$, all edges
have nonpositive weights, and the weight of each edge of~$\tilde{\digr}$ is not smaller than that
of any edge of~$W$.
Each edge of~$W$ is either removed, kept or replaced by an edge of~$\tilde{\digr}$ in~$\tilde{W}$, thus
we conclude that
$p(\tilde{W})\geq p(W)$. This shows
\begin{equation*}
p\left(
\walkslennode{i}{j}{t}{\digr}\setminus 
\walkslennode{i}{j}{t}{\crit(A)}\right)\le \max_l p\left(\walkslennode{i}{j}{t,\gamma_l}{\subcrit_l}\right).
\end{equation*}
However, Theorem~\ref{t:representation} implies that
\begin{equation*}
p\left(\walkslennode{i}{j}{t,\gamma_l}{\subcrit_l}\right)=p\left(\walkslennode{i}{j}{t,\gamma}{\crit_l(A)}\right)
\end{equation*}
for each $l$ and hence
\begin{equation*}
\max_l p\left(\walkslennode{i}{j}{t,\gamma_l}{\subcrit_l}\right)=\max_l 
p\left(\walkslennode{i}{j}{t,\gamma}{\crit_l(A)}\right)= p\left(\walkslennode{i}{j}{t,\gamma}{\crit(A)}\right),
\end{equation*}
and this concludes the proof.
\end{proof}

\if{

It remains to show that $p(\tilde{W})\geq p(W)$. For this assume
without loss of generality that $l(\Tilde{W})=l(W)$: if necessary, $\Tilde{W}$
can be
extended to the length $l(W)$ by adjoining some critical
edges\footnote{Then $\Tilde{W}$ is not a walk anymore, but this does
not matter for the count of weight.}. Comparing $\Tilde{W}$ with $W$
we see that some of the edges of $\Tilde{W}$ are the ``old'' edges
of $W$, and the rest of the edges of $W$ (with weight $\leq\mu(W)$)
have been replaced with the critical edges or the edges of
$\Tilde{\digr}$ (with weight $\geq\mu(W)$). Hence $p(\tilde{W})\geq
p(W)$ as claimed.
}\fi

\if{
\begin{figure}
\caption{Walk~$\tilde{W}$ in proof of Lemma~\ref{l:SdPHA}}
\label{fig:p:CSRsimple:thresh}
\end{figure}
}\fi

Proposition~\ref{p:TcrToT1} part (i) now follows from
Lemmas~\ref{l:SiP} and~\ref{l:SdPHA}.

\subsection{Cycle threshold scheme}
In this section, we perform step~\ref{SdP} of the strategy in the case $B=\bct$.
We prove the following lemma.
\begin{lemma}\label{l:SdPct}
Let $A$ be a square matrix with $\lambda(A)=0$.

For any $t\ge \max\left\{\tilde{T}^{l(\cycle)}_{cr}(\cycle)|\cycle\textnormal{
cycle of $\ctgr$}\right\}$ and any~$i,j$,
inequality~\eqref{e:SdP} holds with~$\gamma=\gamma(\crit(A))$
and~$\digr=\ctgr$ the graph defining~$\bct$ in
Section~\ref{s:csr}.
\end{lemma}


A finite sequence of cycles $\cycle_1,\ldots,\cycle_m$ in~$\subcrit$
is called a {\em staircase\/} in~$\subcrit$ if, for
all~$1\leq s\leq m-1$, $\cycle_s$ and~$\cycle_{s+1}$ share a
node, $p(\cycle_s) / l(\cycle_s) \leq p(\cycle_{s+1}) /
l(\cycle_{s+1})$ and, moreover, the cycle mean of
$\cycle_{s+1}$ is the greatest among all the cycles sharing a
node with $\cycle_{s}$.

\begin{lemma}
\label{l:staircase}
Let~$\mu > \mu^{ct}$ and $\cycle$ be a cycle in~$\thrct(\mu)$
or~$\mu=\mu^{ct}$ and $\cycle$ be a cycle in~$\ctgr(\mu)$ with
$p(\cycle) / l(\cycle) = \mu$.
Then there exists a staircase
$\cycle_1,\ldots,\cycle_m$ in $\thrct(\mu)$
such that $\cycle_1=\cycle$ and $\cycle_m$ is critical.
\end{lemma}
\begin{proof}
Suppose by contradiction that no such staircase exists.
Let~$\cycle_1,\ldots,\cycle_m$ be a staircase in~$\thrct(\mu)$ such that
 $\cycle_1=\cycle$
and $p(\cycle_m)/l(\cycle_m)$ is
maximal.

Denote $\mu' =  p(\cycle_m)/l(\cycle_m)$, so $\mu'<\lambda(A)$.
If the s.c.c.\ of $\thrct(\mu')$, in which
$\cycle_m$ lies, contains a cycle of mean weight strictly greater
than $\mu'$, then we can build a staircase with a greater cycle mean
of the final cycle, a contradiction. So that component of
$\thrct(\mu')$ does not contain a cycle of mean weight strictly
greater than~$\mu'$, which is a contradiction to the definition
of~$\mu^{ct}$ and the fact that $\mu'\geq\mu^{ct}$. Thus we must
have $\mu'=\lambda(A)$.
\end{proof}

\bigskip

\begin{proof}[Proof of Lemma~\ref{l:SdPct} and
Proposition~\ref{p:TcrToT1} part (ii)]

Let $t\geq \max_\cycle \tilde{T}_{cr}(\cycle)$
and let~$W\in\walkslen{i}{j}{t}$ visiting a node of~$\ctgr$ but no critical node.

Denote by~$\nu(W)$ the largest cycle mean of subcycles of~$W$. We
assume in the following that~$\nu(W)$ is maximal among
all~$W\in\walkslen{i}{j}{t}$ with~$p(W) = a^{(t)}_{ij}$. We prove
Lemma~\ref{l:SdPct} by showing $\nu(W)=\lambda(A)$. Assume
that~$\nu(W)<\lambda(A)$, and define
\begin{equation}
\label{tildedigr-ct}
\Tilde{\digr}:=
\begin{cases}
\ctgr & \text{if $\nu(W)\leq\mu^{ct}$}\enspace,\\
\thrct(\nu(W)) & \text{otherwise\enspace.}
\end{cases}
\end{equation}
By the definition of cycle threshold graphs,
$\crit(A)\subseteq\Tilde{\digr}\subseteq\ctgr$.


By Lemma~\ref{l:staircase}, there exists a staircase
$\cycle_1,\dots,\cycle_m$ in~$\tilde{\digr}$ such that~$\cycle_1$
has $p(\cycle_1)=\nu(W)$ and shares a node with~$W$, and~$\cycle_m$
is critical. We inductively define walks~$W_0,\dots,W_m$ as follows:
Set $W_0=W$. For~$1\leq\ell\leq m$, let~$\subcrit$ be the subgraph
of~$\digr(A)$ induced by~$\cycle_\ell$. By definition
of~$\tilde{T}_{cr}$, there is a
walk~$V\in\walkslennode{i}{j}{t,l(\cycle_{\ell})}{\cycle_{\ell}}$ obtained
from~$W_{\ell-1}$ by removing at least one cycle and inserting at
least one cycle in~$\subcrit$ (i.e., one copy of~$\cycle_\ell$) such
that $l(V)\le\tilde{T}^{l(\cycle_\ell)}_{cr}(\cycle_\ell)\le t$. Now
define~$W_\ell$ as walk~$V$ after inserting enough copies
of~$\cycle_\ell$, to have $l(W_\ell)=t$.  Thus $\cycle_\ell$ is a subwalk
of~$W_\ell$ for all $\ell$, and walk~$W_m$ contains a critical node.

We now show that $p(W_\ell)\geq p(W_{\ell-1})$ on each step. For
this we will prove by induction that, on each step, the mean weight
of $\cycle_{\ell+1}$ is not less than that of any cycle (and hence
closed walk) in $W_{\ell}$. The base of induction ($\ell=0$) is due
to the definition of $\tilde{\digr}$. In general, observe that the
cycles in $W_{\ell}$ are 1)$\cycle_{\ell}$ and cycles using the
edges of $\cycle_{\ell}$, 2) cycles that were already in
$W_{\ell-1}$. For the latter cycles we use the inductive assumption,
while the cycles using edges of $\cycle_{\ell}$ share a common node
with it and hence their mean weight does not exceed that of
$\cycle_{\ell+1}$ by the definition of staircase.

Setting $\tilde{W}=W_m$ we obtain $\tilde{W}\in
\walkslennode{i}{j}{t}{\crit(A)}$ and $p(\tilde{W})\geq p(W)$, thus
Lemma~\ref{l:SdPct} and Proposition~\ref{p:TcrToT1} part~(ii) are
proved.
\end{proof}

\section{Cycle Removal}\label{s:Tcr}


\subsection{Cycle removal threshold}
In this section, we state some bounds on $T_{cr}^\gamma(\subcrit)$
for some subgraphs~$\subcrit$ of~$\digr(A)$. Those bounds are
achieved by three different methods, one of them is new. Recall
$\circumf(\digr(A)$, $\cabdrive(\digr(A))$ and other parameters
(Section~\ref{s:statements}).

Let us first recall an elementary application of the pigeonhole
principle. The origins of this lemma were briefly discussed by
Aigner and Ziegler~\cite{AZ:01}, p.~133. In the context of
max-algebraic matrix powers, it was considered for the first time by
Hartmann and Arguelles~\cite{HA-99}. It is in the heart of almost
all of our cycle reductions.

\begin{lemma}
\label{l:NT} Let $a_1,\ldots, a_m$ be integers. Then there exists a
nonempty subset $I\subseteq\{1,\dots,m\}$ of indices such that the
sum $\sum_{i\in I} a_i$ is a multiple of $m$.
\end{lemma}

One of the bounds that we use is in fact proved in~\cite{CBFN-11} (see also \cite{CBFN-12}, Theorem 2).
The proof is recalled for the reader's convenience.
\begin{proposition}{(Lemma~20 of~\cite{CBFN-11})}
\label{p:cbfn}
For any~$A\in\Rpnn$, any node~$i$ and any integer~$\gamma$, we have:
\begin{equation}
\label{e:cbfn}
T_{cr}^\gamma(\{i\})\le (\gamma-1)\circumf+(\gamma+1)\cabdrive,
\end{equation}
where $\cabdrive=\cabdrive(\digr(A))$ and
$\circumf=\circumf(\digr(A))$.
\end{proposition}
\begin{proof}
Let $W$ be a walk going through $i$.  Write this walk as $W=W_0\cdot
\cycle_1\cdot\cdots\cdot\cycle_m\cdot W_m$ where (i) all $\cycle_s$
are nonempty cycles, (ii) node $i$ is a node of the walk $W_r$, and
(iii) $m$ is maximal. Write also $W_r=V_0V_1$ so that $i$ is the end
of $V_0$ and the start of $V_1$. The whole configuration is shown on
Figure~\ref{fig:lem:cbfn}.

If a subset $S\subseteq\{1,\ldots,m\}$ of indices such that $\gamma$
divides $\sum_{s\in S} l(\cycle_s)$ cannot be chosen, then by
Lemma~\ref{l:NT} $m<\gamma-1$, and the walks
$W_1,\ldots,W_{r-1},V_0,V_1,W_{r+1},\ldots,W_m$ are paths (otherwise
$m$ is not maximal), which implies that $l(W)\le
(\gamma-1)\circumf+(\gamma+1)\cabdrive$.

If $l(W)> (\gamma-1)\circumf+(\gamma+1)\cabdrive$, then such a
subset of cycles can be chosen, and a strictly shorter subwalk of
the same length modulo $\gamma$ is obtained by cycle deletion, hence
the claim.
\end{proof}

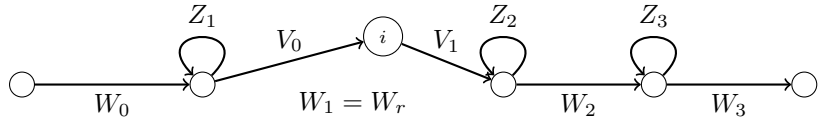
\begin{figure}[ht]
\centering
\begin{tikzpicture}[scale=0.8]
    \node[shape=circle,draw] (i) at (-6,0) {};
    \node[shape=circle,draw] (k) at (0,.8) {$\scriptstyle i$};
    \node[shape=circle,draw] (j) at (7,0) {};
    \node[shape=circle,draw] (n1) at (4.5,0) {};
    \node[shape=circle,draw] (n2) at (2,0) {};
    \node[shape=circle,draw] (n3) at (-3,0) {};
        \node  at (-0.5,-0.3) {$W_1=W_r$};
        \draw[thick,->] (i) -- node[below] {$W_0$} (n3);
        \draw[thick,->] (n3) -- node[above] {$V_0$} (k);
        \draw[thick,->] (k) -- node[above] {$V_1$} (n2);
        \draw[thick,->] (n2) -- node[below] {$W_2$} (n1);
        \draw[thick,->] (n1) -- node[below] {$W_3$} (j);
        \draw[thick,->] (n3) .. controls +(1,1) and +(-1,1) .. node[above] {$\cycle_1$} (n3);
        \draw[thick,->] (n2) .. controls +(1,1) and +(-1,1) .. node[above] {$\cycle_2$} (n2);
        \draw[thick,->] (n1) .. controls +(1,1) and +(-1,1) .. node[above] {$\cycle_3$} (n1);
\end{tikzpicture}
\caption{Walk $W$ in the proof of Proposition~\ref{p:cbfn} ($m=3$ and $r=1$).}
\label{fig:lem:cbfn}
\end{figure}

Proposition~\ref{p:cbfn} implies that $T_{cr}^{l(\cycle)}(\cycle)\le
(l(\cycle)-1)\circumf+(l(\cycle)+1)\cabdrive$ and
$T_{cr}^{\gamma(\crit)}(\crit)\le 2 \gamma(\crit)(n-1)
+\gamma(\crit)-1$ for any s.c.c.\ $\crit\subset\crit(A)$ but both
bounds can be improved using various methods.

The first bound is improved in Section~\ref{s:TcrHA}, following a method used in~\cite{HA-99}, which leads to:
\begin{proposition}\label{p:TcrHA}
For $A\in\Rp^{n\times n}$, $\cycle$ a cycle of~$\digr(A)$ and~$\gamma$ a divisor of~$l(\cycle)$, we have:
\begin{equation}\label{e:TcrHA}
T_{cr}^{\gamma}(\cycle)\le (n-1-l(\cycle)+\gamma)\circumf+\cabdrive+ l(\cycle),
\end{equation}
where $\cabdrive=\cabdrive(\digr(A))$ and $\circumf=\circumf(\digr(A))$.
\end{proposition}

This method also leads to:
\begin{proposition}\label{p:TcrHAWielandt}
For $A\in\Rp^{n\times n}$ and~$\cycle$ a cycle with length~$n$ of~$\digr(A)$, we have $\tilde{T}_{cr}^{n}(\cycle)\le n^2-n+1$.
\end{proposition}

The second bound, is improved in Section~\ref{s:TcrLin} thanks to a new method, which leads to:
\begin{proposition}\label{p:TcRLin}
For $A\in\Rpnn$ and $\subcrit$ a subgraph of~$\digr(A)$ with~$n_1$ nodes, we have:
$$\forall \gamma\in\Nat, T_{cr}^\gamma(\subcrit)\le \gamma n +n-n_1-1.$$
\end{proposition}
\old{
\begin{corollary}\label{c:PumpingLemGene}
For $A\in\Rp^{n\times n}$ and any cycle~$\cycle$ of~$\digr(A)$, we have
$$T_{cr}^{l(\cycle)}\le l(\cycle)(n-2)+n+l(\cycle)-1$$
\end{corollary}
}


Tables~\ref{t:Tcr1} and~\ref{t:Tcr2} show the bounds obtained for a critical
cycle~$\cycle$ or a s.c.c $\crit$ of~$\crit(A)$. Here the first
column contains proposition number, $\gamma=\gamma(\crit)$, and
other parameters refer to~$\digr(A)$. Note that $l(\cycle)=n$ in the
case of Proposition~\ref{p:TcrHAWielandt}.


\begin{table}[h]
\centering
\begin{tabular}{|c|c|c|}
\hline
Prop.&$T^{l(\cycle)}_{cr}(\cycle)-l(\cycle)+1$ & $\tilde{T}^{l(\cycle)}_{cr}(\cycle)$
\\\hline
\ref{p:cbfn} &
$(l(\cycle)-1)(\circumf-1)+(l(\cycle)+1)\cabdrive$&
$(l(\cycle)-1)\circumf+(l(\cycle)+1)\cabdrive+1$
\\\hline

\ref{p:TcrHA}
&$(n-1)\circumf+\cabdrive+1$
&$n\circumf+\cabdrive+1$

\\\hline

\ref{p:TcrHAWielandt} &$\wiel(n)$&$n^2-n+1$

\\\hline
\ref{p:TcRLin}&$l(\cycle)(n-2)+n$ &$l(\cycle)(n-1)+n$
\\\hline
\end{tabular}
\caption{Expressions of Proposition~\ref{p:TcrToT1} (with $l(\cycle)$)}
\label{t:Tcr1}
\end{table}

\begin{table}[h]
\centering
\begin{tabular}{|c|c|}
\hline
Prop.& $T^{\gamma(\crit)}_{cr}(\crit)-\gamma(\crit)+1$
\\\hline
\ref{p:cbfn} & $(\gamma-1)(\circumf-1)+(\gamma+1)\cabdrive$
\\\hline
\ref{p:TcRLin} &$\gamma (n-1) +n-|\crit|$
\\\hline
\end{tabular}
\caption{Expressions of Proposition~\ref{p:TcrToT1} (with $\gamma$)}
\label{t:Tcr2}
\end{table}

\begin{proof}[Proof of Theorems~\ref{t:T1ha} and~\ref{t:T1ct}]
Theorems~\ref{t:T1ha} and~\ref{t:T1ct} are combinations of the
bounds in Tables~\ref{t:Tcr1} and~\ref{t:Tcr2} with
Proposition~\ref{p:TcrToT1}. For each s.c.c.~$\crit$ of~$\crit(A)$,
Table~\ref{t:T1} explains which choices of~$\mN$, $\gamma$ and
proposition to bound $T_{cr}^\gamma(\mN)$ should be made.

\begin{table}[h]
\centering
\begin{tabular}{|c|c|c|c|}
\hline
Bound on~$T_1(A,B)$&$\mN$& $\gamma$ & Prop.
\\\hline
\eqref{e:T1haWielandt} & $\cycle$ s.t. $l(\cycle)=\g(\crit)$ & $l(\cycle)=\g(\crit)$
&\ref{p:TcrHAWielandt}, \ref{p:TcRLin}
\\\hline
\eqref{e:T1haDM} &  $\cycle$ s.t. $l(\cycle)=\g(\crit)$ & $l(\cycle)=\g(\crit)$ &\ref{p:TcRLin}

\\\hline
\eqref{e:T1haCB} & $i\in \cycle$ s.t. $l(\cycle)=\g(\crit)$ &
$l(\cycle)=\g(\crit)$ &\ref{p:cbfn}

\\\hline
\eqref{e:T1haDMEp} & $\crit$ & $\gamma(\crit)$ &\ref{p:TcRLin}

\\\hline
\eqref{e:T1haCBEp} & $i\in\crit$ & $\gamma(\crit)$ &\ref{p:cbfn}

\\\hline

\eqref{e:T1ctCD} & $\cycle$ in staircase or $\cycle$ critical  & $l(\cycle)$
&\ref{p:TcrHA}, \ref{p:TcRLin}

\\\hline

\eqref{e:T1ctCB} & any~$i$ in any~$\cycle$ & $l(\cycle)$ &\ref{p:cbfn}

\\\hline

\end{tabular}
\caption{How to deduce the bounds on $T_1$} \label{t:T1}
\end{table}

To obtain bounds~\eqref{e:T1haWielandt}--\eqref{e:T1haCB} we
take, for the representing subgraph $\subcrit$ in
Proposition~\ref{p:TcrToT1}, any collection of critical cycles
such that each s.c.c.\ of $\crit(A)$ contains exactly one cycle
of the collection and each cycle has the minimal length in the
corresponding s.c.c. In the case
of~\eqref{e:T1haDMEp} and \eqref{e:T1haCBEp}, we set
$\subcrit=\crit(A)$. Bounds~\eqref{e:T1ctCD}
and~\eqref{e:T1ctCB} can be obtained from the last column of
Table~\ref{t:T1}. Note that~\eqref{e:T1ctCD} is obtained as the
minimum of two bounds.

The only difficult case is bound~\eqref{e:T1ctW}. Indeed, in the
worst case, cycle~$\cycle$ with length~$n$, we only get
$\tilde{T}^n_{cr}(\cycle)\le n^2-n+1$ by
Proposition~\ref{p:TcrHAWielandt}. instead of $\wiel(n)$. Thus,
Proposition~\ref{p:TcrToT1} would give $T_1\le n^2-n+1$ instead
of $T_1\le \wiel(n)$ and we have to go into more details. The
proof of~\eqref{e:T1ctW} is thus postponed to the end of the
next subsection.
\end{proof}

\subsection{Cycle removal by cycle decomposition}\label{s:TcrHA}
In this section, we present and improve the method of~\cite{HA-99} to prove Propositions~\ref{p:TcrHA} and~\ref{p:TcrHAWielandt}.
It will also be used to prove that~$T_1(A,\bct)\le \wiel(n)$ (Equation~\eqref{e:T1ctW}) at the end of the next subsection.
For any set of walks $W_{\alpha}$ with $\alpha\in S$ for $S$
a subset of natural numbers, let us denote by $\subcrit(\cup_{\alpha\in S} W_{\alpha})$
the subgraph of $\digr(A)$ consisting of
all nodes and edges that belong to some walk $W_{\alpha}$, $\alpha\in S$.

\begin{proof}[Proof of Propositions~\ref{p:TcrHA}
and~\ref{p:TcrHAWielandt}]

To any walk~$W\in\walksnode{i}{j}{\cycle}$, we apply the following
procedure, adapted from~\cite{HA-99}.
\begin{enumerate}
 \item We choose a decomposition of the walk~$W\in\walks{i}{j}$ into
a path $P$ and a set of cycles $\cycle_\alpha$ for $\alpha\in S$ (with $S$ a subset of
natural numbers). Note that~$P$ may be empty. If it is, walk~$W$ is closed.
Then, it has the same start and end node.

We denote by $n_W$ the number of nodes that appear at least once in~$W$ and by
$\cabdrive_W$ the maximum length of an acyclic walk whose edges belong to $W$.


\item We take a subset~$R_1$ of~$S$ with $|R_1|\le n-l(\cycle)$
such that $\subcrit(P\cup\cycle\cup_{\alpha\in R_1}\cycle_\alpha)$ is connected and contains all nodes appearing in~$W$.
This is possible because the connection of~$\subcrit(P\cup \cycle)$ with
all the nodes of $W$ can be ensured by adding at most $n-l(\cycle)$ edges of $W$
to $P\cup \cycle$, and hence by adding to it at most $n-l(\cycle)$ cycles $\cycle_\alpha$, for $\alpha\in S$.

\item Let~$R_2$ be a result of recursively removing from $S\setminus R_1$ sets of indices
whose corresponding cycles have a combined length that is a multiple
of~$\gamma$.

By Lemma~\ref{l:NT}, $\lvert R_2\rvert \leq \gamma-1$. Let~$R$ be $R_1\cup R_2$.

Set~$\circumf_W=\max_{\alpha\in R} l(\cycle_\alpha)$
(circumference of the walk $W$).

 \item If  $\subcrit_0=\subcrit(P\cup \bigcup_{\alpha\in R} \cycle_\alpha)$ is
connected, then we build a walk~$V\in\walksnode{i}{j}{\cycle}$ by starting from~$P$ and successively inserting
(in some order) all cycles
$\cycle_\alpha$ with~$\alpha\in R$.

 \item Otherwise, we build $V\in\walksnode{i}{j}{\cycle}$ by starting from~$P$ and successively inserting (in some order)
all cycles $\cycle_\alpha$ with~$\alpha\in R$, and~$Z$.
\end{enumerate}

By construction, $l(V) \equiv l(W) \pmod \gamma$ in both cases.
Let us bound the length of~$W$.

If~$\subcrit_0$ is connected,
\begin{equation}
\label{e:M0connect}
\begin{split}
& l(V)  = l(P) + \sum_{\alpha\in R} l(\cycle_\alpha)
\\
& \leq \cabdrive_W + \circumf_W (n-l(\cycle)+\gamma -1)=\circumf_W (n-l(\cycle)+\gamma -2) +(\cabdrive_W + \circumf_W)
\end{split}
\end{equation}

If $\subcrit_0$ is not connected, we have $l(V)\le \circumf_W (n-l(\cycle)+\gamma -2) +(\cabdrive_W + \circumf_W) +l(\cycle)$.

But there is some~$\hat{\alpha}\in R$ such that $l(P) + l(\cycle_{\hat{\alpha}})
\leq n_W-1$, because otherwise every~$\cycle_\alpha$ with $\alpha\in
R$ would share a node with~$P$. Because $\lvert R\setminus
\{\hat{\alpha}\}\rvert \leq n-l(\cycle)+\gamma -2$, we have
%
\begin{equation}
\label{e:M0disconnect}
\begin{split}
l(V) & = l(\cycle) +  l(P) + l(\cycle_{\hat{\alpha}}) + \sum_{\substack{\alpha\in
R\\ \alpha \neq \hat{\alpha}}} l(\cycle_\alpha)
\\
& \leq l(\cycle) + (n_W - 1) + (n-l(\cycle)+\gamma -2)\circumf_W
\end{split}
\end{equation}

Finally, we have $l(V)\le l(\cycle) + (n-l(\cycle)+\gamma -2)\circumf_W
+\min(n_W-1,\circumf_W+\cabdrive_W)$ if $M_0$ is not connected, and
$l(V)\le l(\cycle) + (n-l(\cycle)+\gamma -2)\circumf_W
+(\circumf_W+\cabdrive_W-l(\cycle))$.

This gives the following
\begin{lemma}\label{l:TcrHA}
For any cycle~$\cycle$, any divisor~$\gamma$ of~$l(\cycle)$ and any walk~$W\in\walksnode{i}{j}{\cycle}$,
there is a walk~$V\in\walksnode{i}{j}{\cycle}$ with length at most $l(\cycle) + (n-l(\cycle)+\gamma -2)\circumf_W +\max(\min(n_W-1,\circumf_W+\cabdrive_W),\circumf_W+\cabdrive_W-l(\cycle))$
obtained by removing cycles from~$W$ and possibly inserting~$\cycle$ such that
$l(V) \equiv l(W) \pmod \gamma$.

Moreover, if no copy of~$\cycle$ is inserted then $l(V)\le \circumf_W
(n-l(\cycle)+\gamma -1) + \cabdrive_W $
\end{lemma}

Using that~$\circumf_W\le\circumf(\digr(A))$ and~$\cabdrive_W\le\cabdrive(\digr(A))$, we get Proposition~\ref{p:TcrHA}.\\

\if{ When $l(\cycle)=\gamma=n$, $R_1$ is empty and the cycles in~$R_2$
have length at most~$n-1$ (otherwise they would be removed). So
$\circumf_W\le n-1$ and $\circumf_W+\cabdrive_W-l(\cycle))\le n-2$, and
$l(V)\le (n-2)(n-1)+n-1=n^2-n+1$. This is a bound
for~$\tilde{T}_{cr}$ rather than~$T_{cr}$ because if~$l(W)\ge
n^2-n+1$ and no cycle was inserted, then $l(V)\le
(n-1^2+n-1<n^2-n+1\le t$ and cycles were removed that can be
replaced by~$\cycle$. }\fi

When $l(\cycle)=\gamma=n$, $R_1$ is empty and the cycles in~$R_2$
have length at most~$n-1$ (otherwise they would be removed). So
we use~\eqref{e:M0connect} with $\circumf_W\le n-1$, and we
obtain $l(V)\leq (n-1)(n-1)+n-1=n^2-n$. Hence $T_{cr}^n(\cycle)\leq
n^2-n$ and $\tilde{T}_{cr}(\cycle)\leq n^2-n+1$.
Proposition~\ref{p:TcrHAWielandt} is proved.
\end{proof}

\subsection{Cycle removal by arithmetic method}\label{s:TcrLin}
In this section, we present a new method to bound~$T_{cr}$ leading to Proposition~\ref{p:TcRLin}.

We begin with:
\begin{lemma}\label{LemModg}
Let $\gamma\in\Nat$ and let
$W\in\walks{i}{j}$.
Then there exists a walk $W'\in\walks{i}{j}$ obtained from~$W$ by removing
cycles such that
$l(W')\equiv l(W)\pmod {\gamma}$ and each node appears at most $\gamma$~times
in~$W'$.
\end{lemma}

\begin{proof}
Consider $W$ as a sequence of adjacent nodes $(i_0,\cdots,i_L)$,
where $L$ is the length of the walk.

If a given node appears twice, first as $i_a$ and then as $i_b$ and
if $a\equiv b\pmod {\gamma}$, then the subwalk
$(i_0,\cdots,i_a,i_{b+1},\cdots, i_L)$ is strictly shorter than $W$ and has
the same length modulo~$\gamma$.

Iterating this process, we get a sequence of subwalks of~$W$. Since
the sequence of length is strictly decreasing, the sequence is
finite and we denote the last walk by~$W'$.

Obviously, $l(W')\equiv l(W)\pmod {\gamma}$ and a node does appear twice
as~$i_a$ and $i_b$ only if $a\not\equiv b\pmod {\gamma}$, so
the pigeonhole principle implies that it appears at most $\gamma$~times
(otherwise there would exist $i_a$ and $i_b$ with $a\equiv b\pmod {\gamma}$).
\end{proof}

\begin{proof}[Proof of Proposition~\ref{p:TcRLin}]
We take~$W\in\walksnode{i}{j}{\subcrit}$ and construct a subwalk~$V$ with length at most~$\gamma n +n-n_1-1$ by the following steps.

1. Find the first occurrence of a node of $\subcrit$ in $W$, and
denote this node by~$k$. Let $W_1$ be the subwalk of $W$ connecting
$i$ to $k$, and let $W_2$ be the remaining subwalk. So we have
\begin{equation}\label{Eq:W1W2}
W_1\in\walks{i}{k} ,~ W_2\in\walks{k}{j} ,~ l(W_1)+l(W_2)=l(W)
\end{equation}

2. As long as there is a node $\ell$ that appears twice in~$W_1$ and
at least once in $W_2$, we can write $W_1=U_1\cdot U_2\cdot U_3$ and
$W_2=V_1\cdot V_2$, where $U_1,U_2,V_1$ end with~$\ell$ and
$U_2,U_3,V_2$ start with~$\ell$. Thus, we can replace $W_1$ by
$U_1\cdot U_3$ and $W_2$ by $V_1\cdot U_2\cdot V_2$.
Equation~\eqref{Eq:W1W2} still holds, but now~$i$ appears only once
in~$W_1$. Step~2 is over when all nodes that appear more than once
in~$W_1$ do not appear in~$W_2$. Let us denote the resulting walks
by~$W_3$ and~$W_4$ respectively.

3. Apply Lemma~\ref{LemModg} to $W_3$ and $W_4$, obtaining
$W'_1$ and~$W'_2$ respectively.

4. Set $V=W'_1\cdot W'_2$.

Obviously,
$l(V)\equiv l(W_1)+l(W_2)\equiv l(W)\pmod {\gamma}$.

Now we take a node of $V$ and bound the number of its appearances.

\begin{itemize}
\item[(1)] If it is a node of~$\subcrit\setminus\{k\}$, then it appears only in~$W'_2$, thus at most $\gamma$~times.
$k$ appears once in~$W'_1$, as ending node, and at most $\gamma$
times in~$W'_2$. In the concatenation of the walks, one
occurrence disappears, so all nodes of~$\subcrit$ appear at
most $\gamma$~times.

\item[(2)] If it is a node of~$W'_2$, then it is also a node of~$W_4$, it appears at most once in~$W_3$, thus also in~$W'_1$.
Therefore it appears at most $\gamma+1$~times in~$V$.

\item[(3)] If it is not a node of~$W'_2$, then it appears only in~$W'_1$, thus at most $\gamma$~times.
\end{itemize}

The total number of appearances of all nodes in~$V$ is at most
$(\gamma+1)(n-n_1)+\gamma n_1=\gamma n+(n-n_1)$, so $l(V)$ is bounded by
$\gamma n+(n-n_1-1)$, as claimed.
\end{proof}


%
%
\begin{proof}[Proof of Bound~\eqref{e:T1ctW}]
To prove bound~\eqref{e:T1ctW}, we apply Lemma~\ref{l:SiP} as before
and the difficulty only comes from Lemma~\ref{l:SdPct} that is not
good enough.

To prove that inequality~\eqref{e:SdP} holds with~$t\ge\wiel(n)$
and~$\digr=\ctgr$, we do as in the proof of
Proposition~\ref{p:TcrToT1}: we remove cycles from a walk~$W$ to
replace them by cycles with greater weight, following the staircase
given by Lemma~\ref{l:staircase}. In this process, the walks to
reduce have no critical node.


We apply Lemma~\ref{l:TcrHA} with $\gamma=l(\cycle)$ and we obtain
\begin{equation*}
\begin{split}
l(V)&\le (n-2)\circumf_W+l(\cycle)+\max(n_W-1,\circumf_W+\cabdrive_W-l(\cycle))\\
&\le n_W-1+n=(n-1)n_W+n-1.
\end{split}
\end{equation*}
Since $W$ has no critical node, $n_W\le n-1$, and this bound is less
than~$\wiel(n)$ except when $n_W=n-1$.

But in this last case one has a critical loop on the only
critical node and the rest of the nodes are in $W$. Let $\cycle$ be
the penultimate cycle of the staircase, it shares nodes with $W$
and contains the unique critical node. The weight of this cycle
is greater than or equal to that of all cycles in $W$. Applying
Proposition~\ref{p:TcRLin} with~$\subcrit=\cycle$ and~$\gamma=1$, it
is possible to reduce the walk to a length at most~$2n-l(\cycle)-1$,
insert~$\cycle$ and then as many critical loops as necessary to get
back to a walk with length~$t$.

This is possible if $t\ge 2n-1$. Thus, Equation~\eqref{e:T1ctW} holds true for any~$n$.
\end{proof}

\section{Proof of Theorems~\ref{t:T2} and~\ref{t:T2v}}\label{s:TcrToT2}

Theorem~\ref{t:T2} follows from the bounds on~$T_{cr}$ together with
the following proposition.
\begin{proposition}\label{p:TcrToT2}
Let $A$ be an irreducible matrix, $\subcrit$ be a representing
subgraph of~$\crit(A)$ with cyclicity~$\gamma$ and $B$ be
subordinate to~$A$ such that~$\lambda(B)\neq \bzero$. Then
$$T_2(A,B)\le \frac{T_{cr}^\gamma(\subcrit)(\lambda(A)
-\min_{kl}a_{kl})
+(\max_{kl}b_{kl}-\lambda(B))\cabdrive(\digr(B))}{\lambda(A)-\lambda(B)}.$$

If moreover $A$ has only finite entries, then equations \eqref{e:T2fin}~and~\eqref{e:T2finSyK} hold.
\end{proposition}

We begin with the following lemmas.
\begin{lemma}\label{l:FiniteEntriesEnough}
Let $A\in\Rpnn$ be an irreducible matrix, and $C,S,R$ be defined
relative to any completely
reducible~$\subcrit\subseteq\crit(A)$. For any~$B$ subordinate
to~$A$ and any $t$, if $b_{ij}^{(t)}$ is finite, then
$\left(CS^tR\right)_{ij}$ is finite too.
\end{lemma}
\begin{proof}
If $b_{ij}^{(t)}$ is finite, so is~$a^{(t)}_{ij}$,
By the optimal walk interpretation~\eqref{e:walksense1}, there is a walk $W$
connecting $i$ to
$j$, of length $t$, such that $p(W)=a_{ij}^{(t)}$. As $A$ is
irreducible, there is a closed walk $V$ containing $i$ and a
node $k$ of $\subcrit$. If $\gamma$ is the cyclicity of $\subcrit$
then $V^{\gamma}W\in\walkslennode{i}{j}{t,\gamma}{\subcrit}$,
and $(CS^tR)_{ij}\neq\bzero$ by~\eqref{e:representation}.
\end{proof}

\begin{lemma}\label{l:BndsBt}
For any~$B\in\Rpnn$ and any~$t\in\Nat$, let $\tilde{B}$ be $B-\lambda(B)$, we have:
$$B^t\le t\lambda(B)\otimes \tilde{B}^* \textnormal{ and }
\tilde{b}_{ij}^*\le\cabdrive(\digr(B))\left(\max_{kl}b_{kl}-\lambda(B)\right)
\le(n_B-1)\|B\|$$

If $B$ has only finite entries, then $\tilde{b}_{ij}^*\le (\lambda(B)-\min_{kl}b_{kl})$.
\end{lemma}
\begin{proof}
The first part of the claim immediately follows from the optimal
walk interpretation~\eqref{e:walksense1} and~\eqref{e:walksense2}.

For the second part, observe that $\tilde{b}^*_{ij}$ is equal to
$p(W)-\lambda(B)l(W)$ for some walk $W$ connecting $i$ to $j$ in
$B$. As $b_{ji}\neq\bzero$ we have
$p(W)\leq\lambda(B)(l(W)+1)-b_{ji}$, hence
$\tilde{b}^*_{ij}\leq\lambda(B)-b_{ji}$ and the second part of the
claim.
\end{proof}

\begin{lemma}\label{l:BndsCSR}
Let $A\in\Rpnn$ be a matrix with~$\lambda(A)=\1$, $C,S,R$ be defined relatively to
$\crit(A)$, let $\subcrit$ be a representing subgraph of~$\crit(A)$
and $\gamma$ be a multiple of the cyclicity of~$\subcrit$.

For any~$t\in\Nat$, the finite entries of~$CS^tR$ satisfy
\begin{equation}\label{e:CsrMin}
 (CS^tR)_{ij}\ge T^{\gamma}_{cr}(\subcrit)\min_{kl}a_{kl}
\end{equation}

If $A$ has only finite entries, then for all~$i,j$ we have:
\begin{align}
 (CS^tR)_{ij}&\ge 2\min_{ij}a_{ij}\label{e:CsrMinFiniteA}\\
 (CS^tR)_{ij}&\ge 2\min_{ij}a_{ij}+\tilde{b}^*_{ij}+\cabdrive_B\lambda(B)\label{e:CsrMinFiniteASyK}\\
 (CS^tRv)_{i}&\ge \min_{ij}a_{ij}+\min_jv_j\label{e:CsrvMinFiniteA}
\end{align}
\end{lemma}

Before proving this lemma, let us state another one to use for the matrices with finite entries.
\begin{lemma}\label{l:nonnegativep}
Let $A$ be a matrix with~$\lambda(A)=\1$, then for any integer~$m$ there is a walk~$W_0$ with length~$m$ and nonnegative weight on~$\digr(A)$.
\end{lemma}

\begin{proof}
Let~$\cycle$ be a critical cycle of~$A$. Since
$l(\cycle^m)=l(\cycle)m$, there are walks~$W_1,\cdots,W_{l(\cycle)}$ of
length~$m$ such that $W_1\cdots W_{l(\cycle)}=\cycle^{m}$. Since
$\sum_l p(W_l)=p(\cycle)t=0$, there is a $W_k$
with nonnegative $p(W_k)$.
\end{proof}

\begin{proof}[Proof of Lemma~\ref{l:BndsCSR}]
We first show inequality~\eqref{e:CsrMin}. By the optimal walk
interpretation~\eqref{e:representation} we have
$(CS^tR)_{ij}=\max\{p(W)\colon W\in\walkslennode{i}{j}{t,\gamma}{\subcrit}\}$ for any walk $W$. If
$(CS^tR)_{ij}$ is finite then the walk set
$\walkslennode{i}{j}{t,\gamma}{\subcrit}$ is non-empty and
contains a walk with the length bounded by
$T_{cr}^{\gamma}(\subcrit)$, hence~\eqref{e:CsrMin}.


To prove inequality~\eqref{e:CsrMinFiniteA}, let us assume that
$A$~has only finite entries, and that $t\geq 2+n$ (using that the
sequence $\{CS^tR\}_{t\geq 1}$ is periodic). 

Apply Lemma~\ref{l:nonnegativep} with~$m=t-2$
and set~$W=(i,r)\cdot W_0\cdot
(s,j)$, where $r$, resp.\ $s$, are the beginning node, resp.\ the
end node of $W_0$. By the optimal walk
interpretation~\eqref{e:representation}, we get $(CS^tR)_{ij}\ge
p(W)\ge a_{ir}+a_{sj}\ge 2\min_{kl}a_{kl}$.

The inequalities~\eqref{e:CsrMinFiniteASyK} and~\eqref{e:CsrvMinFiniteA} are proved similarly.
For~\eqref{e:CsrMinFiniteASyK}, select a walk~$V$ with minimal length among those with weight~$\tilde{b}^*_{ij}$ on~$\digr(\tilde{B})$ and a walk~$W_0$
with nonnegative $p(W_0)$ and length~$t-l(V)-2$.
Set~$W=(i,r)\cdot W_0\cdot (s,i)\cdot V$ and get
$$(CS^tR)_{ij}\ge p(W)\ge a_{ir}+a_{si} +p(V)\ge 2\min_{k\ell}a_{k\ell} +\tilde{b}^*_{k\ell} +\lambda(B)\cabdrive_B.$$

For~\eqref{e:CsrvMinFiniteA}, select a  walk $W_0$ with nonnegative
$p(W_0)$ and length~$t-1$ and set $W=(i,r)\cdot W_k$ (where $r$ is the
beginning node of $W_0$).
\end{proof}

\old{
To prove Equation~\eqref{e:CsrvMin}, let us define a walk~$V$ as follows.
Take a walk $V_1$ from~$i$ to~$\subcrit$ with minimal length. Denote by~$s$ the remainder of the division of~$t-l(V_1)$ by~$\gamma$.
Starting from the end node of~$V_1$, follow $s$ critical edges and call this walk~$V_2$.
Set $V=V_1V_2$.
By definition $l(V_1)\le \cabdrive\left(\digr(A)\setminus\subcrit\right)+1$ and~$l(V_2)\le\gamma(\subcrit)-1$, which implies~\eqref{e:CsrvMin}.
}

\begin{proof}[Proof of Proposition~\ref{p:TcrToT2} and Theorem~\ref{t:T2}]
Assume that $\lambda(A)=0$ and $t$ is greater than one of the bounds.
We want to prove that equation
\begin{equation}\label{e:T2}
t\lambda(A)\otimes(CS^tR)_{ij}\ge t\lambda(B)\otimes \tilde{b}^{(t)}_{ij}
\end{equation}
holds for all $i,j$.

By Lemma~\ref{l:FiniteEntriesEnough}, if $(CS^tR)_{ij}=\bzero$ then
$\tilde{b}^t_{ij}=\bzero$ and there is nothing to prove. So we can
assume that $(CS^tR)_{ij}$ is finite, in which case we can use the
inequalities of Lemmas~\ref{l:BndsCSR} and~\ref{l:BndsBt}, which
show that~\eqref{e:T2} follows when we have
\begin{equation}
\label{e:tlAtlB}
\begin{split}
t\lambda(A)+T^{\gamma}_{cr}(\subcrit)\left(\min_{kl}a_{kl}-\lambda(A)\right)&\ge
t\lambda(B)+\cabdrive(\digr(B))(\max_{kl} b_{kl}-\lambda(B)),\\
t\lambda(A)+ 2(\min_{kl}a_{kl}-\lambda(A))&\ge
t\lambda(B)+(\lambda(B)-\min_{kl} b_{kl}),
\end{split}
\end{equation}
in the general case (the first inequality) and in the case of finite
entries (the second inequality). If $t$ is greater than one of the
required bounds, then one of the inequalities~\eqref{e:tlAtlB}
holds, and~\eqref{e:T2} follows.

To obtain Theorem~\ref{t:T2} it remains to deduce the shorter
parts of~\eqref{e:T2a}-\eqref{e:T2c} from the longer ones.
Observe that all the longer parts of the bounds are of the form
\begin{equation}
\label{e:genfrac}
\frac{n_1(\lambda(A)-a_{ij})+\cabdrive_B(a_{kl}-\lambda(B))}{\lambda(A)-\lambda(B)}
\end{equation}
for some $i,j,k,l$, where $n_1$ is greater than $\cabdrive_B$.
Using $n_1>\cabdrive_B$, expression~\eqref{e:genfrac} can be bounded by
\begin{equation*}
\begin{split}
&\frac
{(n_1-\cabdrive_B)(\lambda(A)-a_{ij})+\cabdrive_B(a_{kl}-a_{ij}+\lambda(A)-\lambda(B))}{\lambda(A)-\lambda(B)}\\
&\leq
\frac{(n_1-\cabdrive_B)\|A\|+\cabdrive_B(\|A\|+\lambda(A)-\lambda(B))}{\lambda(A)-\lambda(B)}=
n_1\frac{\|A\|}{(\lambda(A)-\lambda(B)}+\cabdrive(\digr(B)).
\end{split}
\end{equation*}
This completes the proof of all the bounds of
Theorem~\ref{t:T2}.
\end{proof}

It remains to prove~Theorem~\ref{t:T2v}. We do it by generalizing the proof
of \cite[Proposition 5]{CBFN-12}.

\begin{proof}[Proof of Theorem~\ref{t:T2v}]
The case~$\lambda(A)=\bzero$ is trivial.
In the rest of the prove, we assume $\lambda(A)=\bunity$ by replacing~$A$
with~$\lambda(A)^-\otimes A$.

We denote by~$\Delta$ and~$\delta$ the greatest and smallest edge
weight in~$\digr(A)$, respectively. We have $\lVert A\rVert = \Delta
- \delta$. If $\Delta=\delta$, then~$\crit(A)=\digr(A)$ and hence $B^t\le A^t\le
CS^tR$ by the optimal walks interpretations~\eqref{e:walksense1}
and~\eqref{e:representation}.

We hence assume $\Delta\neq\delta$ in the rest of the proof.
The assumption~$\lambda(A)=0$ implies $\delta\leq\lambda(B)\leq 0\leq
\Delta$.

Denote by~$v_{\max}$ and~$v_{\min}$ the greatest and smallest entry of~$v$,
respectively.
It is $\lVert v\rVert = v_{\max} - v_{\min}$.

Let $t\geq {(\lVert v\rVert+ (n-1)\lVert
A\rVert)}/{(-\lambda(B))}$.
We show $CS^tRv\geq
B^tv$.

Let~$i$ be a node of~$\digr(A)$. Let $V$ be a walk in $\digr(B)$ of
length~$t$ starting at~$i$, and let $\Tilde{V}$ be the remaining
walk after repeated cycle deletion. Let $W_2$ be a shortest path
connecting some node~$k'$ of $\Tilde{V}$ to a critical node~$k$ and
let~$W_1$ be the prefix of~$\Tilde{V}$ ending at~$k'$ and let
$\Tilde{V} = W_1\cdot W_1'$. See Figure~\ref{fig:proof:trans:2:w}
for an illustration of these walks. We obtain
\begin{equation}
\label{eq:proof:trans:2:upper}
\begin{split}
p(V) & \leq p(\Tilde{V}) + \lambda(B) \cdot \big( t - l(\Tilde{V}) \big)
\leq
p(W_1\cdot W_1') - \delta \cdot l(W_1\cdot W_1') + \lambda(B)\cdot t
\\ & \leq
p(W_1) + \lVert A\rVert \cdot l(W_1') - \delta \cdot l(W_1) -\|v\| -
\lVert A \rVert \cdot (n-1),
\end{split}
\end{equation}
using that $\lambda(B)t\leq -(||v||+||A||(n-1))$.

Let~$\cycle$ be a critical cycle starting at~$k$ and set $r =
\left\lfloor \big( t - l(W_1\cdot W_2)  \big) / l(\cycle)
\right\rfloor$. Then let~$W_3$ be the prefix of~$\cycle$ of length
$t - l(W_1\cdot W_2\cdot \cycle^r)$, which is between~$0$ and
$l(\cycle) - 1$. Setting $W = W_1\cdot W_2\cdot \cycle^r \cdot W_3$,
we have
\begin{equation}\label{eq:proof:trans:2:lower}
p(W) \geq p(W_1\cdot W_2\cdot W_3)
\geq
p(W_1) + \delta\cdot l(W_2\cdot W_3)
\end{equation}
and hence, because $l(W_1) + l(W_2) + l(W_3) + l(W_1') \le n-1$,
\begin{align}
p(V)  & \le
p(W_1) - \lVert A\rVert \cdot \big( l(W_1) + l(W_2) + l(W_3) \big) - \delta \cdot l(W_1)-\|v\|
\\ & \leq
p(W_1) + \delta \cdot l(W_2 \cdot W_3)  -\|v\|
\leq p(W)-\|v\|
\enspace.
\end{align}
Since $p(W)+v_{\min}\leq (CS^tRv)_i$ by the walk
interpretation~\eqref{e:representation} of CSR terms, we have
$(CS^tRv)_{i}\ge (B^tv)_{i}$, which concludes the proof.

The claim for the case that all entries of~$A$ are finite follows from
Lemmas~\ref{l:BndsBt} and~\ref{l:BndsCSR}.
\end{proof}

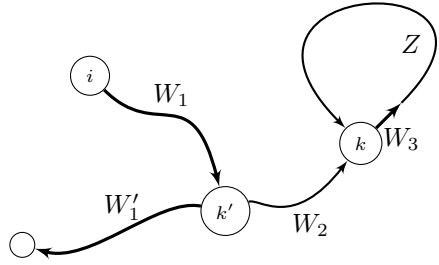
\begin{figure}
\centering
\begin{tikzpicture}[>=latex',scale=0.90]
    \node[shape=circle,draw] (i) at (-2,2) {$\scriptstyle i$};
    \node[shape=circle,draw] (k) at (0,0) {$\scriptstyle k'$};
    \node[shape=circle,draw] (j) at (2,1) {$\scriptstyle k$};
    \node[shape=circle,draw] (end) at (-3,-.5) {};
    \draw[very thick,->] (i) .. controls (-1,1) and (-.5,2)  ..
node[midway,above]{${W_1}$}  (k);
    \draw[thick,->] (k) .. controls (0.5,0.2) and (1,-0.3)  .. node[near end,below=1mm]{${W}_2$}  (j);
    \draw[very thick,->] (j) --  (2.6,1.6) node[below=2mm] {$W_3$};
    \draw[thick,->] (2.6,1.6) .. controls +(2.1,2.1) and +(-2.5,2.5)  ..
node[near start,below left]{$\cycle$}  (j);
    \draw[very thick,->] (k) .. controls +(-1,0.2) and +(1,-0.3) ..
node[above]{$W_1'$} (end);
\end{tikzpicture}
\caption{Walks~$\Tilde{V}=W_1W'_1$ and~$W=W_1W_2\cycle^rW_3$ in proof of Theorem~\ref{t:T2v}.}
\label{fig:proof:trans:2:w}
\end{figure}

\section{Cycle Insertion}\label{s:Ep}
In this section, we state some bounds on~$\ep^\gamma$.

The exploration penalty has been introduced in~\cite{CBFN-12}, where the following is proven.
\begin{proposition}[Theorem~3 of \cite{CBFN-12}]
Let~$\subcrit$ be a strongly connected graph with cyclicity~$\gamma$ and girth~$\g$.
Its exploration penalty~$\ep^\gamma$ satisfies:
$$\ep^\gamma\le 2 \frac{\g}{\gamma} |\subcrit|-\frac{\g}{\gamma}-2\g+\gamma$$
\end{proposition}

Since $\ep^\gamma(\subcrit)$ is bounded by $\ind(\subcrit)$, it is also possible to use the following bounds.
\begin{proposition}
\label{p:booleanbounds}
Let~$\subcrit$ be a strongly connected
graph. Its index~$\ind(\subcrit)$ is related to its girth~$\g$ and
its cyclicity~$\gamma$ by the following inequalities:
\begin{itemize}
\item[{\rm (i)}]
\label{Wielandt} $\ind(\subcrit)\le \wiel(|\subcrit|)$, where
$\wiel(1)=0$ and $\wiel(r)=(r-1)^2+1$ otherwise.
\item[{\rm (ii)}]
\label{Schwarz} $\ind(\subcrit)\le \gamma \wiel(r)+s$, where $r$
is the quotient of the division of $|\subcrit|$  by $\gamma$ and
$s$ its remainder.
\item[{\rm (iii)}]
\label{DulmageMendelsohn} $\ind(\subcrit)\le
|\subcrit|+(|\subcrit|-2)\g$
\end{itemize}
\end{proposition}

Bound (i) can be traced back to a work of Wielandt~\cite{Wie-50}.
Bound (ii) is due to Schwarz~\cite{Sch-70}, but a more comprehensive
explanation was given by Shao and Li~\cite{LS-93}. Bound (iii) was
originally proved by Dulmage and Mendelsohn~\cite{DM-64} for
primitive matrices but the case of a non-primitive matrix also
follows (for instance) from Theorem~\ref{t:T1ha} by
Remark~\ref{r:Boolean2}. Other bounds on $\ind(\digr)$ can be also
found in the literature.

As noticed by Kim~\cite{Kim-80}, the same method as in the proof
of~\eqref{Schwarz} by Shao and Li~\cite{LS-93} applied
to~\eqref{DulmageMendelsohn} instead of~\eqref{Wielandt} gives:
$$\ind(\subcrit)\le \gamma r+(r-2)\g + s\le |\subcrit|(1+\frac{\g}{\gamma})-2\g.$$

Bouillard and Gaujal~\cite{BG-00} implicitly derive
from~\eqref{DulmageMendelsohn} the following:
\begin{proposition}
If $\crit(A)$ has $n_c$ nodes, $h$~s.c.c.'s and maximal girth~$\hat{\g}$, then any s.c.c.~$\crit$ of~$\crit(A)$ satisfies
$\ep^{\gamma(\crit)}(\crit)\le n_c+(n_c-2h)\hat{\g}$.
\end{proposition}

Indeed, if the s.c.c.~$\crit_l$ of~$\crit(A)$ has $n_l$ nodes, for
any node $i$ in~$\crit_l$, we have
$$\ep^{\gamma(\crit_l)}(i)
\le\ind(\crit_l)
\le\sum_{l=1}^h\ind(\crit_l)
\le\sum_{l=1}^h \left(n_i+(n_i-2)\hat{\g}\right)
\le n_c+(n_c-2h)\hat{\g}.$$
Together with this bound, Equation~\eqref{e:T1haDMEp} generalizes and improves the bound of~\cite{BG-00}.

\section{Local Reductions}\label{s:localred}


Every weak CSR expansion gives rise to {\em local weak CSR expansions} that can take the following forms:

\begin{equation}
\label{e:weaklocal}
\begin{split}
&a^{(t)}_{ij}=(CS^tR)_{ij}\oplus b^{(t)}_{ij},\quad\text{for $t\geq\tilde{\tau}(i,j)$},\\
&a^{(t)}_{ij}=(CS^lR)_{ij}\oplus b^{(t)}_{ij},\quad\text{for $t\equiv l(\text{mod }\gamma)$ and $t\geq\tilde{\tau}(i,j,l)$},\\
& (A^tv)_i = (CS^lRv)_i\oplus
(B^tv)_i,\quad\text{for $t\geq\tilde{\tau}(i,v)$},\\
& (A^tv)_i = (CS^lRv)_i\oplus
(B^tv)_i,\quad\text{for $t\equiv l(\text{mod}\ \gamma)$ and $t\geq\tilde{\tau}(i,l,v)$},
\end{split}
\end{equation}

In connection with these schemes, define the following subsets:
\begin{equation}
\label{e:Js}
\begin{split}
& J(i,j):=\{s\colon a^*_{is}a^*_{sj}<\min_l (CS^lR)_{ij}\},\\
& J(i,j,l):=\{s\colon a^*_{is}a^*_{sj}< (CS^lR)_{ij}\},\\
& J(i,v):=\{s\colon \bigoplus_j a^*_{is}a^*_{sj}v_j<\min_l
(CS^lRv)_i\}\\
& J(i,l,v):=\{s\colon \bigoplus_j a^*_{is}a^*_{sj}v_j<(CS^lRv)_i\}\\
\end{split}
\end{equation}
\begin{remark}
Unless $i=s=j$, $a_{is}^*a_{sj}^*$ is the biggest weight of a walk connecting $i$ to $j$ via $s$.
It follows from Theorem~\ref{t:representation} and this optimal walk interpretation
that $i,j\notin J(i,j),\ J(i,j,l)$ and $i\notin J(i,v),\
J(i,l,v)$. Moreover, if some critical~$s$ belongs to one of the sets
defined here, then all its s.c.c.\ in~$\crit(A)$ does, since for each pair of nodes in the same s.c.c. of $\crit(A)$
we can find a closed walk in $\crit(A)$ containing both of them.
\end{remark}

Note that $i,j\notin J(i,j),\ J(i,j,l)$ and
$i\notin J(i,v),\ J(i,l,v),$ by the optimal walk interpretation~\eqref{e:walksense2} of
$A^*$ and CSR terms~\eqref{e:representation}.

Now let $\Tilde{\crit}(A)$ be the remainder of the critical graph,
without the s.c.c.\ with indices in $J$, for $J=J(i,j)$, $J(i,j,l)$,
$J(i,v)$ or $J(i,l,v)$.

Redefine $\Tilde{C},\Tilde{S}$ and $\Tilde{R}$ using
$\Tilde{\crit}(A)$ instead of $\crit(A)$, and the subordinate
matrix' $\Tilde{A}$ of $A$ where all rows and columns with indices
in $J$ are canceled, instead of $A$. Redefine $\Tilde{B}$ as a
subordinate of $\Tilde{A}$ whose indices are in $\digr(B)$ (but not
in $J$). This procedure will be referred to as {\em local reduction}
of a weak CSR expansion. When $J=J(i,j)$, or resp.\ $J=J(i,j,l)$,
$J=J(i,v)$ or $J=J(i,v,l)$, this will be called $i,j$-reduction, or
resp.\ $i,j,l$-reduction, $i,v$-reduction or $i,l,v$-reduction.


\begin{theorem}
\label{t:weakreduced}
Let $A\in\Rpnn$, $B$ subordinate to $A$ and the integer numbers
$\tilde{\tau}(i,j)$, $\tilde{\tau}(i,j,l)$, $\tilde{\tau}(i,v)$  and
$\tilde{\tau}(i,l,v)$, for $i,j\in\{1,\ldots,n\}$ and $v\in\Rp^n$,
satisfy~\eqref{e:weaklocal}.
Corresponding to the definitions of $J$ given in~\eqref{e:Js}, we
have
\begin{equation}
\label{e:weakreduced}
\begin{split}
&a^{(t)}_{ij}=(\Tilde{C}\Tilde{S}^t\Tilde{R})_{ij}\oplus \tilde{b}^{(t)}_{ij},\quad\text{for $t\geq\tilde{\tau}(i,j)$},\\
&a^{(t)}_{ij}=(\Tilde{C}\Tilde{S}^l\Tilde{R})_{ij}\oplus \tilde{b}^{(t)}_{ij},\quad\text{for $t\equiv l(\text{mod }\gamma)$ and $t\geq\tilde{\tau}(i,j,l)$},\\
& (A^tv)_i = (\Tilde{C}\Tilde{S}^l\Tilde{R}v)_i\oplus
(\Tilde{B}^tv)_i,\quad\text{for $t\geq\tilde{\tau}(i,v)$},\\
& (A^tv)_i = (\Tilde{C}\Tilde{S}^l\Tilde{R}v)_i\oplus
(\Tilde{B}^tv)_i,\quad\text{for $t\equiv l(\text{mod }\gamma)$ and
$t\geq\tilde{\tau}(i,l,v)$},
\end{split}
\end{equation}
with $\Tilde{C},\Tilde{S},\Tilde{R}$ and $\Tilde{B}$ defined in the local
reduction procedure.
\end{theorem}

\begin{proof}
We prove the theorem only in the case of $i,j$-reduction, i.e., in
the first case of~\eqref{e:weakreduced} corresponding to the first
case of~\eqref{e:weaklocal} and~\eqref{e:Js}.  The rest is
analogous. Let $N_B$, resp. $N_c$ be the set of nodes of $\digr(B)$,
resp. $\crit(A)$.

\if{
Define the subordinate matrix' $A'$ of $A$ formed by setting
to~$\bzero$ all rows and columns with indices in $J(i,j)\cap
N_B$. We first show that the first equation of~\eqref{e:weaklocal}
for $a_{ij}^{(t)}$ holds also with CSR terms and $B$ defined from
$A'$ instead of $A$. First, notice that the weights of walks going
through $s\in J(i,j)\cap N_B$ are less than $\min_l (CS^lR)_{ij}$,
hence $(CS^lR)_{ij}$ (being the weight of a walk connecting $i$ to
$j$) cannot decrease for any $l$ when defined from the subordinate
matrix' $A'$, so it is exactly the same. Next, the weight of any
walk going through any $s\in J(i,j)\cap N_B$ is dominated by any CSR
term, that is, $p(\walkslennode{i}{j}{t}{s})<\min_l (CS^lR)_{ij}$,
and we can replace $b^{(t)}_{ij}$ by $\Tilde{b}^{(t)}_{ij}$ in the
first equation of~\eqref{e:weaklocal}.
}\fi

Define the subordinate matrix $A'$ of $A$ formed by setting
to~$\bzero$ all rows and columns with indices in $J(i,j)\cap
N_B$. We first show that the first equation of~\eqref{e:weaklocal}
for $a_{ij}^{(t)}$ holds also with CSR terms and $B$ defined from
$A'$ instead of $A$. First, recall that the weights of walks going
through $s\in J(i,j)\cap N_B$ are less than $\min_l (CS^lR)_{ij}$,
and $(CS^lR)_{ij}$ is the greatest weight of any walk with certain length constraint, connecting $i$ to
$j$ via a critical node. Defining $(CS^lR)_{ij}$ from $A'$ amounts to canceling all walks going through
$s\in J(i,j)\cap N_B$ and contributing to $(CS^lR)_{ij}$.
Since such walks have low weight, $(CS^lR)_{ij}$ does not decrease, for any $l$, when defined from the subordinate
matrix $A'$, so it is exactly the same. Next, 
observe that (since the weights of walks going
through $s\in J(i,j)\cap N_B$ are less than $\min_l (CS^lR)_{ij}$) 
we can replace $b^{(t)}_{ij}$ by $\Tilde{b}^{(t)}_{ij}$ in the
first equation of~\eqref{e:weaklocal}.

We next show that the CSR term defined from $A'$ can be reduced. Use
expansion~\eqref{e:early-exp} of the CSR terms defined from $A'$,
where the first terms in~\eqref{e:early-exp} are defined from the
components of $\crit(A)$ with indices in $J(i,j)$ (these components
can be taken in any order). The sum of these terms expresses
$p(\walkslennode{i}{j}{t}{J(i,j)\cap N_c})$ (with walk sets defined
in $\digr(A')$) for all large enough $t$. Since these walk weights
are strictly less than the entries of CSR, all those terms in
expansion~\eqref{e:early-exp} with indices in $J(i,j)$ can be
canceled. The remaining part of expansion (for the entry $i,j$) sums
up to the reduced CSR term defined from the subordinate matrix
$\Tilde{A}$ (as defined in the reduction procedure).
\end{proof}

Notice that the proofs of the Theorems~\ref{t:T2} and~\ref{t:T2v} in Section~\ref{s:TcrToT2}
work with walks in $\walkslen{i}{j}{t}$. Hence the corresponding
bounds can be combined with all reductions of
Theorem~\ref{t:weakreduced}. In particular, local reductions may
lead to smaller $B$ and $\lambda(B)$ when $i$ and $j$ or $i$ and $v$
are fixed. Moreover, they can also result in decrease of the initial bounds on
$\tilde{\tau}$ based on the cycle removal threshold, since some of the
critical components get removed.

We now also recall a bound of a type that can
be found in Akian, Gaubert and Walsh~\cite{AGW-05}, and Bouillard and Gaujal~\cite{BG-00},
formulated here for the case of $(i,j,l)$-reduction.

\begin{proposition}
\label{p:AGW} Suppose that $A\in\Rp^{n\times n}$ is irreducible,
with $\lambda(A)=1$, and take $i,j\in\{1,\ldots,n\}$ and $l>0$. Let
$\gamma$ be the cyclicity of $\crit(A)$, and let $\tilde{\tau}(i,j,l)$ be
an integer such that
\begin{equation}
\label{CSR+B_ij}
a_{ij}^{(t)}=(\Tilde{C}\Tilde{S}^l\Tilde{R})_{ij}\oplus \Tilde{b}^{(t)}_{ij},\quad
t\equiv l(\text{mod}\ \gamma),\ t\geq\tilde{\tau}(i,j,l)\enspace,
\end{equation}
where the terms $\Tilde{C},\Tilde{S},\Tilde{R}$ and matrix
$\Tilde{B}$ are obtained by the $i,j,l$-reduction of some weak CSR
expansion. Let
\begin{equation}
\label{Tijl}
T(i,j,l)=\min\left\{t\colon \lambda^{\otimes t}(\Tilde{B})\otimes(\lambda^-(\Tilde{B})\otimes\Tilde{B})_{ij}^*\leq
(\Tilde{C}\Tilde{S}^l\Tilde{R})_{ij}\right\}\enspace .
\end{equation}
Then the transient $\tau(i,j,l)$ for which
\begin{equation}
\label{aijt}
a_{ij}^{(t)}=(\Tilde{C}\Tilde{S}^l\Tilde{R})_{ij},\quad t\equiv
l(\text{mod}\;\gamma),\ t\geq\tau(i,j,l)\enspace,
\end{equation}
satisfies $\tau(i,j,l)\leq\max(\tilde{\tau}(i,j,l),T(i,j,l))$.
\end{proposition}
\begin{proof}
We only need to show that $\Tilde{b}_{ij}^{(t)}\leq \lambda^{\otimes
t}(\Tilde{B})\otimes (\lambda^-(\Tilde{B})\otimes\Tilde{B})^*_{ij}$.
Indeed, this follows
after dividing both parts of this inequality (in max-plus sense) by
$\lambda^{\otimes t}(\Tilde{B})$ and using the optimal walk
interpretation of $(\lambda^-(\Tilde{B})\otimes\Tilde{B})^t$ and
$(\lambda^-(\Tilde{B})\otimes\Tilde{B})^*$.
\end{proof}

\section{Acknowledgement}

We would like to thank Marianne Akian, Anne Bouillard, Peter Butkovi\v{c},
Bernadette Charron-Bost, Matthias F\"{u}gger, St\'{e}phane Gaubert, Rob Goverde,
Bernd Heidergott, Jean Mairesse, and Hans Schneider for many useful discussions, advice, and
inspiration.


\providecommand{\bysame}{\leavevmode\hbox
to3em{\hrulefill}\thinspace}
\providecommand{\MR}{\relax\ifhmode\unskip\space\fi MR }
\providecommand{\MRhref}[2]{%
  \href{http://www.ams.org/mathscinet-getitem?mr=#1}{#2}
} \providecommand{\href}[2]{#2}

\end{document}